\documentclass[11pt]{article}
\usepackage{fullpage}
\usepackage{amsmath}
\usepackage{amssymb}
\usepackage{amsthm}
\usepackage{fancyhdr}
\usepackage{algorithm}
\usepackage{algorithmicx}
\usepackage{algpseudocode}
\usepackage{graphicx}
\usepackage{multirow}
\usepackage{enumitem}
\usepackage{bm}
\usepackage{dsfont}
\usepackage{caption}
\usepackage[usenames]{color}
\allowdisplaybreaks
\newcommand\numberthis{\addtocounter{equation}{1}\tag{\theequation}}

\newtheorem{theorem}{Theorem}[section]
\newtheorem{lemma}[theorem]{Lemma}
\newtheorem{corollary}[theorem]{Corollary}

\newtheorem*{remark*}{Remark}
\def\la{\left\langle}
\def\ra{\right\rangle}
\def\lb{\left(}
\def\rb{\right)}
\def\ln{\left\|}
\def\rn{\right\|}
\def\lab{\left|}
\def\rab{\right|}
\def\lcb{\left\{}
\def\rcb{\right\}}
\def\lsb{\left[}
\def\rsb{\right]}
\def\ba{\bm{a}}
\def\bx{\bm{x}}
\def\by{\bm{y}}
\def\bz{\bm{z}}
\def\bu{\bm{u}}
\def\bb{\bm{b}}
\def\bv{\bm{v}}
\def\bh{\bm{h}}
\def\bq{\bm{q}}
\def\be{\bm{e}}

\def\bd{\bm{d}}
\def\bzero{\bm{0}}
\def\bw{\bm{w}}
\def\btlow{ \vartheta_{\mbox{\scriptsize\normalfont lb}}}
\def\btup{ \vartheta_{\mbox{\scriptsize\normalfont up}}}
\def\epup{ \varepsilon{\mbox{\scriptsize\normalfont up}}}
\def\BI{\bm{I}}
\def\TM{{\mbox{\normalfont\textbf{I}}}}
\def\BA{\bm{A}}
\def\BX{\bm{X}}

\def\BI{\bm{I}}

\def\BZ{\bm{Z}}
\def\BW{\bm{W}}
\def\BF{\bm{F}}

\def\BG{\bm{G}}
\def\N{\mathcal{N}}
\def\A{\mathcal{A}}
\def\ME{\mathcal{E}}

\def\P{\mathcal{P}}
\def\I{\mathcal{I}}
\def\T{\mathcal{T}}

\def\R{\mathbb{R}}
\def\C{\mathbb{C}}
\DeclareMathOperator*{\dist}{\mbox{\normalfont dist}}
\DeclareMathOperator*{\rank}{\mbox{\normalfont rank}}

\newcommand{\mean}[1] {\mathbb{E}{\lsb#1\rsb}}
\newcommand{\prob}[1] {\mathbb{P}{\lsb#1\rsb}}
\def\dsone{\mathds{1}}
\newcommand{\onesub}[1] {\dsone_{\lcb #1\rcb}}

\title{Solving systems of phaseless equations via Riemannian optimization with optimal sampling complexity}
\author{Jian-Feng Cai\thanks{Department of Mathematics, Hong Kong University of Science and Technology, Clear Water Bay, Kowloon, Hong Kong SAR, China.}
\and Ke Wei\thanks{School of Data Science, Fudan University, Shanghai, China.}}

\begin{document}
\maketitle
\begin{abstract}
A Riemannian gradient descent algorithm and a truncated variant are presented to solve systems of phaseless equations $|\BA\bx|^2=\by$. The algorithms are developed by exploiting the inherent low rank structure of the problem based on the embedded manifold of rank-$1$ positive semidefinite matrices. Theoretical recovery guarantee has been established for the truncated variant, showing that the algorithm is able to achieve successful recovery when the number of equations is proportional to the number of unknowns. Two key ingredients in the analysis are the restricted well conditioned property and the restricted weak correlation property of the associated truncated linear operator. 
 Empirical evaluations show that our algorithms are competitive with other state-of-the-art first order nonconvex approaches with provable guarantees. 
\end{abstract}
%-------------------------------------
%-------------------------------------
\section{Introduction}\label{sec:introduction}
In this paper we are interested in finding a vector $\bx\in\R^n/\C^n$ which solves the following system of  phaseless equations:
\begin{align*}
|\BA\bx|^2=\by,\numberthis\label{eq:problem}
\end{align*}
where $\BA\in\R^{m\times n}/\C^{m\times n}$ and $\by\in\R^{m}$ are both known. Compared with linear systems of the form $\BA\bx=\sqrt{\by}$, it is self-evident that after taking the entrywise modulus phase information is missing from \eqref{eq:problem}.  Thus, seeking a solution to \eqref{eq:problem} is often referred to as generalized phase retrieval, extending the classical phase retrieval problem where $\BA$ is a Fourier type matrix to a general setting. Phase retrieval arises in a wide range of practical context such as  X-ray crystallography \cite{Ha93a}, diffraction imaging \cite{Buetal07} and microscopy \cite{Mietal08}, where it is hard or infeasible to record the phase information when detecting an object. Many heuristic yet effective algorithms have been developed for phase retrieval, for example, Error Reduction, Hybrid Input Output and other variants \cite{GeSaPhase,fienup1,fienup2,wavefront}.

Injectivity  has been investigated in \cite{BaCaEd06a,CoEdHeVi2013a}, showing that $m\geq 2n-1$ real generic measurements or $m\geq 4n-4$ complex generic measurements are sufficient to determine a unique solution of \eqref{eq:problem} up to a global phase vector. Despite this, solving systems of phaseless equations is computationally intractable. For simplicity, let us consider the real case. Then one can immediately see the combinatorial nature of the problem since there there are $2^m$ possible signs for $\by$. In fact, a very simple instance of \eqref{eq:problem} is equivalent to the NP-hard stone problem \cite{CC:CPAM:17}.

Over the past few years computing methods with provable guarantees have received extensive  investigations for solving systems of phaseless equations, typically based on the Gaussian measurement model. In a pioneering work by Cand\`es et al. \cite{CSV:CPAM:13}, a convex relaxation via trace norm minimization, known as PhaseLift, was studied. The approach was developed based on the fact that \eqref{eq:problem} can be cast as a rank-$1$ positive semidefinite matrix recovery problem. Inspired by the work on low rank matrix recovery, it was established that under the Gaussian measurement model PhaseLift was able to find the solution of \eqref{eq:problem} with high probability provided that\footnote{The notation $m\gtrsim f(n)$ means that there exists an absolute constant $C>0$ such that $m\ge C\cdot f(n)$.}  $m\gtrsim n\log n$. This sampling complexity was subsequently sharpened to $m\gtrsim n$ in \cite{CL:FCM:14}. The recovery guarantee of PhaseLift under coded diffraction model was studied in \cite{phaselift4,phaselift5}. There were also several other convex relaxation methods for solving systems of phaseless equations; see for example \cite{phasecut,phasemax1,phasemax2,phasemax3}.

Convex methods are amenable to detailed analysis, but they are not computationally desirable for large scale problems. Thus  more scalable yet still provable nonconvex methods have received particular attention recently. A resampled variant of Error Reduction  has been investigated in \cite{alt_min_pr}, showing that $m\gtrsim n\log^3 n+n^2\log^2 n\log(1/\epsilon)$ number of measurements are sufficient for the algorithm to attain an $\epsilon$-accuracy. A gradient descent algorithm called Wirtinger Flow (WF) was developed based on an intensity-based loss function, and it was shown that the algorithm could achieve successful recovery provided $m\gtrsim n\log n$ \cite{CLS:TIT:15}.  A variant of WF, known as Truncated Wirtinger Flow (TWF), was introduced in \cite{CC:CPAM:17} based on the Poisson loss function, which could achieve successful recovery under the optimal sampling complexity $m\gtrsim n$. In \cite{ZCL:ICML:16}, the algorithm was analyzed when  median truncation was used.  Another gradient descent algorithm, termed Truncated Amplitude Flow (TAF), was developed in \cite{WGE:TIT:18} based on an  amplitude-based loss function. Optimal theoretical recovery guarantee of TAF was similarly established under the Gaussian measurement model. In \cite{phase_kacz}, the classical Kaczmarz method for solving  systems of linear equations was extended to solve the generalized phase retrieval problem. The optimal sampling complexity for the successful recovery of the Kaczmarz method was established in \cite{phase_kacz02,phase_kacz03} when the unknown vector and the measurement matrix were both real. 

The nonconvex algorithms mentioned in the last paragraph are all analyzed based on some local geometry and hence closeness of the initial guess to the ground truth is required \cite{CLS:TIT:15,CC:CPAM:17,WGE:TIT:18,phase_kacz02,phase_kacz03}. In contrast, there is a line of research which attempts to study the global geometry of related problems; see \cite{GJZ:ICML:17,SQW:FCM:18,GJZ:ICML:17,GJZ:ArXiv:17} and references therein. Many algorithms have also been designed to utilize the global geometry effectively \cite{GHJY:COLT:15,JGNKJ:ArXiv:17,CDHS16,AAZBHM16}.
We omit further details as it is beyond the scope of this paper and interested
readers are referred to the references.

\paragraph{Main contributions} In this paper we propose  a Riemannian gradient descent algorithm for solving systems of phaseless equations. The algorithm is developed by exploiting the inherent low rank structure of \eqref{eq:problem} based on the embedded manifold of low rank matrices, similar to the Riemannian gradient descent algorithm for the low rank matrix recovery problem studied in \cite{WCCL:SIMAX:16,WCCL:ArXiv:16}. That being said,  there are two key differences. On the algorithmic side, an additional structure, i.e., the positive semidefinite property of the underlying matrix, is incorporated when designing the algorithm. On the theoretical side, since the linear operator related to \eqref{eq:problem}
 does not possess a good concentration around its expectation, it is not very clear how to establish the convergence of the vanilla Riemannian gradient descent algorithm. To overcome the difficulty, we introduce an equally effective truncated variant of the algorithm and theoretical recovery guarantee has been established for this variant. Since the linear operator considered in this paper is substantially different from the one considered in \cite{WCCL:SIMAX:16,WCCL:ArXiv:16}, it is by no means trivial to establish the convergence of the variant. Our work has been greatly influenced by \cite{CC:CPAM:17}, though the key ideas in the analysis are significantly different. Finally, empirical performance evaluations show that our algorithms are competitive with other-state-of the art provable nonconvex gradient descent algorithms.
% is substantially different from the one considered in CITE, it is by no means trivial to establish the convergence of the algorithm. To overcome this difficulty, we introduce an equally effective truncated variant of the algorithm.

\paragraph{Outline and notation}
The remainder of this paper is structured as follows. The Riemannian gradient descent algorithm and its truncated variant is presented in Section~\ref{sec:algs}, together with the exact recovery guarantee for the truncated algorithm. In Section~\ref{sec:numerics} we compared our algorithms with TAF and TWF via a set of numerical experiments.  The proofs of the main results are presented in Section~\ref{sec:proofs1}, with the proofs of the technical lemmas being presented in Section~\ref{sec:proofs2}. We conclude this paper with some potential future directions in Section~\ref{sec:conclusion}.

Throughout the paper we use the following notational conventions. We denote vectors by bold lowercase letters and matrices by bold uppercase letters. In particular, we fix $\bx$ and $\BX$ as the ground truth and its lift matrix (i.e., $\BX=\bx\bx^\top$). We denote by $\|\BZ\|$ and $\|\BZ\|_F$ the spectral norm and Frobenius norm of the matrix $\BZ$, respectively. Restricting to a vector $\bz$, $\|\bz\|$ denotes its $\ell_2$-norm, and the $\ell_1$-norm of $\bz$ is denoted by $\|\bz\|_1$. Operators are denoted by calligraphic
letters, for example, $\A$ denotes a linear operator from $n\times n$ symmetric matrices to vectors of length $m$. Moreover, we use $\tau_x$, $\tau_z$, $\tau_h$ and $\tau_{h,z}$ to denote the truncation parameters, where 
the former three are predetermined and the last one is computed via $\tau_{h,z}=\tau_z+\lb 0.3\tau_h\lb\tau_z+1.2\tau_x\rb+\tau_z^2 \rb^{1/2}$.
\section{Riemannian gradient descent and a truncated variant}\label{sec:algs}
In this section, we present the Riemannian gradient descent algorithm   and its truncated variant for solving systems of phaseless equations.  For ease of exposition, we focus on the real case, but emphasize that the algorithms and the corresponding theoretical results are readily extended to the complex case. 

Let $\ba_k^\top$ denote the $k$-th row of $\BA$, and let $\A$ be a linear operator from $n\times n$ symmetric matrices to vectors of length $m$, defined as \begin{align*}
\A(\BW) = \lcb\langle\BW,\ba_k\ba_k^\top\rangle\rcb_{k=1}^m,\quad\forall~\BW\in\R^{n\times n}\mbox{ being symmetric}.\numberthis\label{eq:A}
\end{align*}
Then a simple algebra yields that 
\begin{align*}
y_k = |\ba^\top_k\bx|^2 = \langle\ba_k\ba_k^\top,\BX\rangle,
\end{align*}
where $\BX=\bx\bx^\top$ is the  lift matrix defined from $\bx$. %Noting that $\BX$ is a rank-$1$ positive semidefinite matrix and 
%$
%[\A(\BX)]_k = |\ba_k^\top\bx|^2=y_k,
%$
%it is not difficult to see that phase retrieval can be cast as the following low rank matrix recovery problem: 
%\begin{align*}
%\min_{\BZ}\rank(\BZ)\quad\mbox{subject to}\quad\A(\BZ)=\by\mbox{ and }\BZ\succeq 0.
%\end{align*}
Noticing the one to one correspondence between $\BX$ and $\bx$, instead of reconstructing $\bx$, one can attempt to reconstruct $\BX$ by seeking a rank-$1$ positive semidefinite matrix which fits the measurements as well as possible: 
\begin{align*}
\min_{\BZ}\frac{1}{2}\|\A(\BZ)-\by\|^2\quad\mbox{subject to}\quad \rank(\BZ)=1\mbox{ and }\BZ\succeq 0.\numberthis\label{eq:low_rank}
\end{align*}
If we parameterize $\BZ$ as $\BZ=\bz\bz^\top$, where $\bz\in\R^{n}$, then the constraints in \eqref{eq:low_rank} can be removed and a gradient descent iteration with respect to $\bz$ leads to the Wirtinger Flow algorithm for solving systems of phaseless equations \cite{CLS:TIT:15}. In this paper we will exploit the low rank structure directly based on the embedded manifold of rank-$1$ and positive semidefinite matrices.
%%%%
\subsection{Riemannian gradient descent}
\begin{algorithm}[ht!]
\caption{Riemannian Gradient Descent (RGrad)}\label{alg:RGrad}
\begin{algorithmic}
\Statex \textbf{Initial guess: }$\BZ_0$.
\For{$l=0,1,\cdots$}\\
\quad 1. $\BG_l  =\A^\top(\by-\A(\BZ_l))$,\\
\quad 2. $\BZ_{l+1}=\T_1(\BZ_l+\alpha_l\P_{T_l}(\BG_l))$.
\EndFor
\end{algorithmic}
\end{algorithm}

It is well-known that the set of fixed rank (e.g., rank-$1$ here) positive semidefinite matrices form a smooth  manifold  when embedded in $\R^{n\times n}$ \cite{HuangGaZh16}.  A Riemannian gradient descent algorithm  (RGrad) based on this embedded manifold structure is described in Algorithm~\ref{alg:RGrad}. Let $\BZ_l$ be the current estimate of $\BX$. RGrad first updates $\BZ_l$ along the projected gradient descent direction $\P_{T_l}(\BG_l)$ with a stepsize $\alpha_l$, where $T_l$ is  the tangent space of the embedded manifold at $\BZ_l$, followed by the projection onto the set of rank-$1$ and positive semidefinite matrices via a thresholding operator  denoted by $\T_1$, which will be discussed in detail later.

In the expression for the gradient descent direction $\BG_l$, $\A^\top$ denotes the adjoint of $\A$, given by 
\begin{align*}
\A^\top(\bb) = \sum_{k=1}^mb_k\ba_k\ba_k^\top, \quad\forall~ \bb\in\R^{m}.\numberthis\label{eq:At}
\end{align*}
Assuming $\BZ_l$ is a rank-$1$ positive semidefinite matrix, then it admits the eigenvalue decomposition $\BZ_l=\sigma_l\bu_l\bu_l^\top$ with $\sigma_l> 0$. The tangent space of the embedded manifold of rank-$1$ positive semidefinite matrices at $\BZ_l$ is given by \cite{HuangGaZh16}
\begin{align*}
T_l =\{\bu_l\bw^\top+\bw\bu_l^\top~|~\bw\in\R^n\}.\numberthis\label{eq:tangent}
\end{align*}
Given an arbitrary  matrix $\BW\in\R^{n\times n}$, the projection of $\BW$ onto $T_l$ can be computed as follows: 
\begin{align*}
\P_{T_l}(\BW) = \bu_l\bu_l^\top\BW+\BW\bu_l\bu_l^\top-\bu_l\bu_l^\top\BW\bu_l\bu_l^\top
\numberthis\label{eq:projT}
\end{align*}
If we decompose a nonzero vector $\bw\neq \bzero$ into a weighted sum of $\bu_l$ and 
$\bv_l$ where $\|\bv_l\|=1$ and $\bu_l\perp\bv_l$ as follows
\begin{align*}
\bw=a\bu_l+b\bv_l,
\end{align*}
then it can be easily seen that each matrix $\BW_l=\bu_l\bw^\top+\bw\bu_l^\top$ in $T_l$ has the following decomposition
\begin{align*}
\BW_l = \begin{bmatrix}
\bu_l & \bv_l
\end{bmatrix}
\begin{bmatrix}
2a & b\\
b& 0
\end{bmatrix}
\begin{bmatrix}
\bu_l^\top\\
\bv_l^\top
\end{bmatrix}.
\end{align*}
A simple algebra reveals that the middle $2\times 2$ symmetric matrix has at least a nonnegative eigenvalue and also at least a nonpositive eigenvalue, so does $\BW_l$ since $\begin{bmatrix}\bu_l & \bv_l\end{bmatrix}$ is an $n\times 2$ orthogonal matrix. Also due to this fact, the eigenvalue decomposition of $\BW_l$ can be constructed very easily from the eigenvalue decomposition of the middle $2\times 2$ matrix. 

In the second step of RGrad, $\T_1$ is a type of retraction in Riemannian optimization \cite{AbMaSe2008manifold} which returns the best rank-$1$ and positive semidefinite approximation of a matrix. Noticing that $\BZ_l+\alpha_l\P_{T_l}(\BG_l)\in T_l$ is of rank at most $2$, the best approximation can be computed by only retaining the larger nonnegative eigenvalue and the corresponding eigenvector in its eigenvalue decomposition.  The stepsize $\alpha_l$ in RGrad can either be   a constant or be computed adaptively. An exact linear search along $\P_{T_l}(\BG_l)$ yields a closed form stepsize given by 
\begin{align*}
\alpha_l = \frac{\|\P_{T_l}(\BG_l)\|_F^2}{\|\A(\P_{T_l}(\BG_l))\|_2^2}.\numberthis\label{eq:RGrad_stepsize}
\end{align*}
%It is evident that $\BZ_l\in T_l$, so is $\BZ_l+\alpha_l\P_{T_l}(\BG_l)$.

Before proceeding, it is worth noting that the Riemannian optimization algorithms based on the embedded manifold of low rank matrices have already beed developed and studied for unstructured low rank matrix recovery problems. In \cite{Van:SIOPT:13},  a Riemannian conjugate gradient descent algorithm was  introduced for  matrix completion. The difference and connection between the Riemannian optimization algorithms based on the embedded manifold of fixed rank $r$ matrices and the iterative hard thresholding algorithms for low rank matrix recovery were pointed out in \cite{KeWeiThesis}, and then exact recovery guarantees of  the corresponding Riemannian gradient descent and conjugate gradient descent algorithms were established in \cite{WCCL:SIMAX:16,WCCL:ArXiv:16} for matrix sensing and  matrix completion respectively. In contrast to the general low rank matrix recovery problem where the target matrix is  low rank but unstructured, the target matrix of interest in this paper is not only low rank but also positive semidefinite.  Thus, in the Riemannian gradient descent algorithm for solving systems of phaseless equations tangent spaces consisting of symmetric matrices are utilized in the design of the algorithm. Meanwhile, the eigenvalue decomposition instead of the singular value decomposition is applied to compute the projection onto the set of rank-$1$ and positive semidefinite matrices. Additionally, the sensing operator $\A$ here is significantly different with that for 
general low rank matrix recovery, which raises a new challenge for the recovery guarantee analysis.
%%%%%
\subsection{Truncated Riemannian gradient descent}
As stated earlier, exact recovery guarantees of Riemannian optimization based on the embedded manifold of low rank matrices have been investigated in \cite{WCCL:SIMAX:16,WCCL:ArXiv:16} for unstructured low rank matrix recovery. A key component in the analysis is the restricted isometry property of the sensing operator. However, even assuming $\ba_k\sim\N(0,\BI_n)$ in the measurement model, the restricted isometry property does not hold for the corresponding sensing operator defined in \eqref{eq:A}. The reason is that, in this case, one can always construct a rank-$1$  matrix whose column space is well-aligned with that of the measurement matrix $\ba_k\ba_k^\top$ for some $1\leq k\leq m$. Thus the incoherence between the underlying low rank matrix and the measurement matrices will be violated; see \cite{CLS:TIT:15} for details. 

\begin{algorithm}[ht!]
\caption{Truncated Riemannian Gradient Descent (TRGrad)}\label{alg:tRGrad}
\begin{algorithmic}
\Statex \textbf{Initial guess: }$\BZ_0$.
\For{$l=0,1,\cdots$}\\
\quad 1. $\BG_l  =\A_l^\top(\by-\A_l(\BZ_l))$,\\
\quad 2. $\BZ_{l+1}=\T_1(\BZ_l+\alpha_l\P_{T_l}(\BG_l))$.
\EndFor
\end{algorithmic}
\end{algorithm}

Inspired by the idea of truncation in  \cite{CC:CPAM:17}, we introduce   a truncated variant of the Riemannian gradient descent algorithm where the sensing operator that is used in each iteration is computed adaptively from $\A$ based on the measurement vector  $\by$ and the current estimate $\BZ_l$; see Algorithm~\ref{alg:tRGrad}. 
Recall that $\by=|\BA\bx|^2$ is the measurement vector. Given any rank-$1$ positive semidefinite matrix $\BZ=\bz\bz^\top$, let $\A_{\bz}$ be the linear operator associated with $\BZ$, defined as
%In particular, letting $\BZ_l=\bz_l\bz_l^\top$,
%the sensing operator $\A_l$ in Algorithm~\ref{alg:tRGrad} is given by (cf. $\A$ in \eqref{eq:A})
\begin{align*}
\A_{\bz}(\BW) = \lcb\langle\BW,\ba_k\ba_k^\top\rangle\dsone_{\ME_1^k(\bx)\cap \ME_1^k(\bz)\cap\ME_2^k(\bz)}\rcb_{k=1}^m,\numberthis\label{eq:Az}
\end{align*}
where $\dsone_{{\bm{\cdot}}}$ is an indicator function, and $\ME_1^k(\bx)$,  $\ME_1^k(\bz)$ and $\ME_2^k(\bz)$ are  three collections of events determining the truncation rules. Here, $\ME_1^k(\bx)$,  $\ME_1^k(\bz)$ and $\ME_2^k(\bz)$ are given by 
\begin{align*}
&\ME_1^k(\bx) = \lcb \sqrt{y_k}\leq \tau_x\sqrt{\frac{\ln\by\rn_1}{m}}\rcb,\numberthis\label{eq:truncation_rules_x1}\\
&\ME^k_1(\bz)=\lcb |\ba_k^\top\bz|\leq \tau_z\|\bz\|\rcb,\numberthis\label{eq:truncation_rules_z1}\\
&\ME^k_2(\bz)=\lcb\lab y_k-|\ba_k^\top\bz|^2\rab\leq\frac{\tau_h}{m}\|\by-\A(\bz\bz^\top)\|_1\frac{|\ba_k^\top\bz|+\sqrt{y_k}}{\ln\bz\rn}\rcb.
\numberthis\label{eq:truncation_rules_z2}\end{align*}
with prescribed truncation parameters $\tau_x$, $\tau_z$ and $\tau_h$. In words, for fixed $\BZ=\bz\bz^\top$, if $\ba_k$ satisfies the truncation rules specified by $\ME_1^k(\bx)$,  $\ME_1^k(\bz)$ and $\ME_2^k(\bz)$, then the $k$-th entry of $\A_{\bz}(\BW)$ is given by $\langle\BW,\ba_k\ba_k^\top\rangle$; otherwise it is set to $0$.
Note that the adjoint of $\A_{\bz}$ is given by
\begin{align*}
\A_{\bz}^\top(\bb)=\sum_{k=1}^m b_k\ba_k\ba_k^\top\dsone_{\ME_1^k(\bx)\cap \ME_1^k(\bz)\cap\ME_2^k(\bz)},\quad\forall~\bb\in\R^m.
\end{align*}
To simplify the notation, we have used $\A_l$ and $\A_l^\top$ to denote $\A_{\bz_l}$ and $\A_{\bz_l}^\top$  respectively in Algorithm~\ref{alg:tRGrad}.
%%%
\subsubsection{Main result: Recovery guarantee of TRGrad}
We begin with an informal discussion on what the truncation rules  can imply. Assume $\ba_k\sim\N(0,\BI_n)$, $k=1,\cdots, m$, are independent. Noting that 
$\sqrt{y_k}=|\ba_k^\top\bx|$ and\footnote{The notation $\asymp$ means the left hand side can be lower as well as upper bounded by a multiple of the right hand side with different universal constants.} $\|\by\|_1/m\asymp\|\bx\|^2$, the first event $\ME_1^k(\bx)$ is equivalent to 
\begin{align*}
\lcb |\ba_k^\top\bx|\lesssim\|\bx\|\rcb\numberthis\label{eq:tmp1}
\end{align*}
Assume $\bz\approx \bx$. We have $\ba_k^\top\bz\approx\ba_k^\top\bx$. {It follows that} $|\ba_k^\top\bz|+\sqrt{y_k}\asymp|\ba_k^\top\bz|$,
\begin{align*}
\lab y_k-|\ba_k^\top\bz|^2\rab &=|\ba_k^\top(\bz+\bx)||\ba_k^\top(\bz-\bx)|\asymp |\ba_k^\top\bz||\ba_k^\top(\bz-\bx)|,\quad\mbox{and}\\
\frac{1}{m}\|\by-\A(\bz\bz^\top)\|_1&\asymp\mean{\lab|\ba_k^\top\bx|^2-|\ba_k^\top\bz|^2\rab}=\mean{|\ba_k^\top(\bz+\bx)||\ba_k^\top(\bz-\bx)|}\\
&\asymp\|\bz\|\|\bz-\bx\|.
\end{align*}
Substituting these into \eqref{eq:truncation_rules_z2}, after canceling common factors, one can roughly reduce $\ME_2^k(\bz)$ to 
\begin{align*}
\lcb |\ba_k^\top(\bz-\bx)|\lesssim\|\bz-\bx\|\rcb.\numberthis\label{eq:tmp2}
\end{align*}
Noticing that \eqref{eq:truncation_rules_z1} has the same form as 
\eqref{eq:tmp1} and \eqref{eq:tmp2}, the truncation rules specified by 
$\ME_1^k(\bx)$,  $\ME_1^k(\bz)$ and $\ME_2^k(\bz)$ basically exclude those components where the measurement matrix $\ba_k\ba_k^\top$ is well-aligned with the column subspaces of the matrices $\BX$, $\BZ$, and $\BZ-\BX$ appearing in the computation. The above arguments will be made precise in the formal proof. Exact recovery guarantee of TRGrad can be established in the following theorem.
%\begin{theorem}[Main result]\label{thm:main}
%Let $\tau_{h,z}=\tau_z+\lb 0.3\tau_h\lb\tau_z+1.2\tau_x\rb+\tau_z^2 \rb^{1/2}$. Assume $\varepsilon_0>0$ is  constant  obeying
%\begin{align*}
%\varepsilon_0\leq \frac{1}{2}\min\lcb\sqrt{\frac{\rho_3}{3\lb \tau_z^4+5\tau_z^3+8\tau_z^2+2\tau_h^2\rb}},\frac{\rho_3}{15\tau_z\tau_{h,z}},\frac{1}{11}\rcb\numberthis\label{eq:basin_size2}
%\end{align*}
%for a sufficiently small constant $\rho_3>0$. Let $\BZ_0=\bz_0\bz_0^\top$ be an initial guess which satisfies \begin{align*}\|\BZ_0-\BX\|_F\leq \varepsilon_0 \|\BX\|_F\numberthis\label{eq:initial_condition}.
%\end{align*}  Then there exists a constant $0<\nu_g<1$  such that the iterates of RGD  converge linearly to $\BX$,
%\begin{align*}
%\|\BZ_{l+1}-\BX\|_F\leq \nu_g\|\BZ_l-\BX\|_F,
%\end{align*}
 %provided the stepsize in each iteration obeys
%\begin{align*}
%\alpha_l\in\lsb\frac{1-\epup}{\btlow-\sqrt{\rho_4\btup}},\frac{1+\epup}{\btup+\sqrt{\rho_4\btup}}\rsb,
%\end{align*}
%where $\epup=\lb 1-\varepsilon_0\sqrt{1+16\varepsilon_0}\rb/\sqrt{1+16\varepsilon_0}$, and $\rho_4$, $\btlow$, $\btup$ are defined in Theorems~\ref{thm:well_conditioned} and \ref{thm:weak_correlation}. 
%Here, we assume 
%$$\btlow>\sqrt{\rho_4\btup}\quad\mbox{and}\quad(1-\epup)\btup+2\sqrt{\rho_4\btup}\leq (1+\epup)\btlow,,$$ which will be met when the truncation parameters are sufficiently large and $\varepsilon_0$ is sufficiently small.
%\end{theorem}
\begin{theorem}[Main result]\label{thm:main}
There exists a numerical constant $\varepsilon_0>0$ (relying on the truncation parameters) such that if the initial guess $\BZ_0=\bz_0\bz_0^\top$ obeys
 \begin{align*}
 \|\BZ_0-\BX\|_F\leq \varepsilon_0 \|\BX\|_F,
 \numberthis\label{eq:initial_condition}
\end{align*}
then with probability exceeding\footnote{The notation $\Omega(m)$ means it is greater than $c\cdot m$ for some constant $c>0$.} $1-e^{-\Omega(m)}$ the iterates of TRGrad with a proper stepsize converge linearly to $\BX$, i.e.,
\begin{align*}
\|\BZ_{l+1}-\BX\|_F\leq \nu_g\|\BZ_l-\BX\|_F,\quad\mbox{for some }0<\nu_g<1,
\end{align*}
provided that $m\gtrsim n$. More precisely, we require that 
\begin{align*}
\varepsilon_0\leq \frac{1}{2}\min\lcb\sqrt{\frac{\rho_3}{3\lb \tau_z^4+5\tau_z^3+8\tau_z^2+2\tau_h^2\rb}},\frac{\rho_3}{15\tau_z\tau_{h,z}}\frac{1}{11}\rcb\numberthis\label{eq:basin_size2}
\end{align*}
 and 
\begin{align*}
\alpha_l\in\lsb\frac{1-\epup}{\btlow-\sqrt{\rho_4\btup}},\frac{1+\epup}{\btup+\sqrt{\rho_4\btup}}\rsb.\numberthis\label{eq:stepsize_range}
\end{align*}
Here $\rho_3>0$ is a sufficiently small constant, $\epup=\lb 1-\varepsilon_0\sqrt{1+16\varepsilon_0}\rb/\sqrt{1+16\varepsilon_0}$, and $\rho_4$, $\btlow$ and $\btup$ are defined in Theorems~\ref{thm:well_conditioned} and \ref{thm:weak_correlation}.
\end{theorem}
In order for the closed interval in \eqref{eq:stepsize_range} to be valid, it  requires that
\begin{align*}
\btlow>\sqrt{\rho_4\btup}\quad\mbox{and}\quad(1-\epup)\btup+2\sqrt{\rho_4\btup}\leq (1+\epup)\btlow.
\end{align*}
From the expressions for the parameters, it is not hard to show that these two conditions can be met for  sufficiently large truncation parameters and sufficiently small $\rho_3$. As a simple result, we can also establish the local convergence of TRGrad with the adaptive stepsize computed via \eqref{eq:RGrad_stepsize} but with $\A$ be replaced with $\A_l$ in each iteration.
\begin{theorem}[Local convergence of TRGrad with steepest descent stepsize]\label{thm:rgd_sd_local}
Let $\BZ_0=\bz_0\bz_0^\top$ be an initial guess obeying $\|\BZ_0-\BX\|_F\leq \varepsilon_0 \|\BX\|_F$, where $\varepsilon_0>0$ satisfies \eqref{eq:basin_size2}.  Then with probability exceeding $1-e^{-\Omega(m)}$ the iterates of TRGrad  with the steepest descent stepsize converge linearly to $\BX$, i.e.,
\begin{align*}
\|\BZ_{l+1}-\BX\|_F\leq \nu_g\|\BZ_l-\BX\|_F, \quad\mbox{for some }0<\nu_g<1,
\end{align*}
 provided that $m\gtrsim n$ and
\begin{align*}
(1-\epup)\btup\leq\btlow-\sqrt{\rho_4\btup}\quad\mbox{and}\quad \btup+\sqrt{\rho_4\btup}\le(1+\epup)\btlow.\numberthis\label{eq:cond_rgd_sd}
\end{align*}
\end{theorem}
Once again, the conditions in \eqref{eq:cond_rgd_sd} can be satisfied for  sufficiently large truncation parameters and sufficiently small $\rho_3$.
%%%
The proofs of Theorems~\ref{thm:main} and \ref{thm:rgd_sd_local} are presented in Section~\ref{sec:proofs1}. A few remarks are in order: 
\begin{itemize}
\item The convergence rate $v_g$ is independent of the ground truth $\bx$ (or $\BX$) and its length $n$, but relies on the truncation parameters. 
\item By Proposition~C.1 in \cite{CC:CPAM:17} (also see Proposition B.1 in \cite{phase_kacz03}), one can use the truncated spectral method to construct an initial vector $\bz_0$ such that for any fixed $\epsilon>0$,
$
\|\bz_0-\bx\|\leq \epsilon\|\bx\|
$ holds with high probability  provided $m\gtrsim \epsilon^{-2}\cdot n$. Letting $\BZ_0=\bz_0\bz_0^\top$, on the same event, we have 
\begin{align*}
\|\BZ_0-\BX\|_F\leq(\|\bz_0\|+\|\bx\|)\|\bz_0-\bx\|\leq (2+\epsilon)\epsilon\|\bx\|^2=(2+\epsilon)\epsilon\|\BX\|_F.
\end{align*}
Thus, the initial condition in \eqref{eq:initial_condition} can be achieved with optimal sampling complexity. It follows that, seeded with this initialization, TRGrad converges linearly to the ground truth target matrix with high probability provided $m\gtrsim \epsilon^{-2}\log\epsilon^{-1}\cdot n$.
%%%
\item The proof strategy for Theorem~\ref{thm:main} is substantially different from the one in \cite{CC:CPAM:17}, though the truncated rules are partially similar. In \cite{CC:CPAM:17}, a local regularity condition of the objective function was established in the vector domain. In contrast, the proof of Theorem~\ref{thm:main} relies on the well conditioned property  and the weak correlation property of the truncated sensing operator $\A_{\bz}$ when being restricted onto low dimensional subspaces; see Theorem~\ref{thm:well_conditioned} and \ref{thm:weak_correlation}. 
%%%
\item Though TRGrad is motivated from the perspective of theoretical analysis, numerical simulations in Section~\ref{sec:numerics} suggest that TRGrad is able to avoid overshooting more effectively than RGrad when using a constant stepsize; see Figure~\ref{fig1}.
\end{itemize}
%%%%
\subsection{Efficient implementations of RGrad and TRGrad}\label{sec:implementation}
since the estimate $\BZ_l$ in RGrad and TRGrad is a rank-$1$ positive semidefinite matrix, we can parameterize it by its eigenvalue decomposition $\BZ_l=\sigma_l\bu_l\bu_l^\top$. Thus, it suffices to update $\sigma_l$ and $\bu_l$
in each iteration. First, it follows from \eqref{eq:projT} that 
\begin{align*}
\BZ_l+\alpha_l\P_{T_l}(\BG_l) & = \sigma_l\bu_l\bu_l^\top+\alpha_l
\lb \bu_l\bu_l^\top\BG_l+\BG_l\bu_l\bu_l^\top-\bu_l\bu_l^\top\BG_l\bu_l\bu_l^\top\rb\\
& = \sigma_l\bu_l\bu_l^\top+\alpha_l\lb\bu_l\bu_l^\top\BG_l\bu_l\bu_l^\top+\lb\BI-\bu_l\bu_l^\top\rb\BG_l\bu_l\bu_l^\top+\bu_l\bu_l^\top\BG_l\lb\BI-\bu_l\bu_l^\top\rb\rb\\
&=\sigma_l\bu_l\bu_l^\top+\alpha_l\lb c_l\bu_l\bu_l^\top+s_l\bv_l\bu_l^\top+s_l\bu_l\bv_l^\top\rb\\
&=\begin{bmatrix}
\bu_l & \bv_l
\end{bmatrix}
\begin{bmatrix}
\sigma_l+c_l & \alpha_ls_l\\
\alpha_ls_l & 0
\end{bmatrix}
\begin{bmatrix}
\bu_l^\top\\
\bv_l^\top
\end{bmatrix},
\end{align*}
where we have used the substitutions $c_l=\bu_l\BG_l\bu_l$ and $\bv_ls_l = \lb\BI-\bu_l\bu_l^\top\rb\BG_l\bu_l$ for $\|\bv_l\|=1$  and $s_l=\|\lb\BI-\bu_l\bu_l^\top\rb\BG_l\bu_l\|$.  Let 
\begin{align*}
\begin{bmatrix}
\sigma_l+c_l & \alpha_ls_l\\
\alpha_ls_l & 0
\end{bmatrix}=\begin{bmatrix}
\bq_1 & \bq_2
\end{bmatrix}
\begin{bmatrix}
\lambda_1 & \\
& \lambda_2
\end{bmatrix}
\begin{bmatrix}
\bq_1^\top\\
\bq_2^\top
\end{bmatrix}
\end{align*}
be the eigenvalue decomposition which can be computed using $O(1)$ flops.
As discussed previously, we must have $\lambda_1\geq 0$ and $\lambda_2\leq 0$, or vice versa. 
Without loss of generality, we can assume $\lambda_1\geq 0$ and $\lambda_2\leq 0$. Then, $\sigma_{l+1}$ and $\bu_{l+1}$ can be updated  by 
\begin{align*}
\sigma_{l+1} = \lambda_1\quad\mbox{and}\quad \bu_{l+1}=\begin{bmatrix}
\bu_l & \bv_l
\end{bmatrix}\bq_1.
\end{align*}

Suppose $\BG_l\bu_l$ is already computed. From the above discussion, it can be easily seen that we can obtain $\sigma_{l+1}$ and $\bu_{l+1}$ using additional $O(n)$ flops.
Hence it only remains to see how to compute $\BG_l\bu_l$.  For RGrad, a simple algebra yields that 
\begin{align*}
\BG_l\bu_l &= \lb\sum_{k=1}^m\ba_k\ba_k^\top\lb y_k-\langle\ba_k\ba_k^\top,\sigma_l\bu_l\bu_l^\top\rangle\rb\rb\bu_l\\
&=\sum_{k=1}^m(\ba_k^\top\bu_l)\lb y_k-\sigma_l|\ba_k^\top\bu_l|^2\rb\ba_k\\
&=\BA^\top\lb (\BA\bu_l)\odot(\by-\sigma_l|\BA\bu_l|^2)\rb,
\end{align*}
where $\odot$ denotes the Hadamard product. Thus, the dominant cost in the computation of $\BG_l\bu_l$ lies in the two matrix-vector products involving $\BA$ and $\BA^\top$ which is $4mn$. For TGRad, $\BG_l\bu_l$ can be computed similarly but with the measurements that are not accepted by the truncation rules being removed. Note that the computational cost to set up the truncation rules is negligible since we can compute $\BA\bz_l$ via $\BA\bz_l=\sqrt{\sigma_l}\BA\bu_l$. In summary, the leading order per iteration cost of RGrad and TGRad is $4mn$ for the computation of  two matrix-vector products involving $\BA$ and $\BA^\top$.

%{\color{red} ABOUT COMPUTING STEPSIZE}
\section{Numerical results}\label{sec:numerics}
In this section we evaluate the empirical performance of RGrad and TRGrad, and compare them with two state-of-the-art first order methods: TWF \cite{CC:CPAM:17} and TAF \cite{WGE:TIT:18}, which are downloaded from the authors' website.  All the algorithms are seeded with an initial guess constructed by the truncated spectral method \cite{CC:CPAM:17}.
The experiments are executed from Matlab 2017b.
%%%
\subsection{Empirical phase transitions}
Here we investigate the recovery ability of the aforementioned algorithms on reconstructing signals of length $n=128$. The signals are generated to have i.i.d Gaussian entries, i.e., $\bx\sim\N(0,\BI_n)$ in the real case and $\bx\sim\N(0,\BI_n)+i\cdot\N(0,\BI_n)$ in the complex case.
Two measurement models are considered: 
\begin{itemize}
\item Gaussian measurement model where $\BA$ has i.i.d Gaussian entries, either real or complex up to whether $\bx$ is real or complex;
\item CDP measurement model \cite{CLS:TIT:15} where $\BA\bx$ is given by 
\begin{align*}
\BA\bx = \begin{bmatrix}
\BF(\bd_1\odot\bx)\\
\vdots\\
\BF(\bd_L\odot\bx)
\end{bmatrix}.\numberthis\label{eq:CDP}
\end{align*}
Here $\BF$ stands for the DFT matrix, and each entry of $\bd_\ell$ ($\ell=1,\cdots,L$) is sampled from $\{1,-1,i,-i\}$ with equal probability. For the CDP model, we also test of the reconstruction of 2D signals whether both $\bx$ and $\bd_\ell$ are matrices of size $128\times 128$, with the 2D DFT being used in \eqref{eq:CDP}.
\end{itemize}
We consider an algorithm to have successfully reconstructed a test signal $\bx$ if 
it returns an estimate $\bx_l$ such that 
$
\dist(\bx_l,\bx)/\|\bx\|\leq 10^{-3},
$
where 
\begin{align*}
\dist(\bx_l,\bx) = \min_{\phi\in[0,2\pi)}\|\bx_l-\bx e^{i\phi}\|.
\end{align*}
Notice that from the efficient implement of RGrad and TRGrad, the estimate $\bx_l$ can be easily formed as $\bx_l=\sqrt{\sigma_l}\bu_l$; see Section~\ref{sec:implementation}.
For each type of measurement models, tests are conducted for $m$ increasing from small value to a sufficiently large one.  For each fixed pair of $(n,m)$, $100$ random simulations are repeated.
Then we calculate the probability of successful recovery out of the $100$ random tests.

We first compare RGrad and TRGrad with a constant stepsize $\alpha_l=0.2$ as well as the adaptive steepest descent stepsize given in \eqref{eq:RGrad_stepsize}, which are labelled (C) and (A) respectively.
The plots of successful recovery probability against the oversampling ratio $m/n$
are presented in Figure~\ref{fig1}. On one hand, RGrad and TRGrad  exhibit similar recovery performance when using adaptive steepest descent stepsize. 
On the other hand, TRGrad has improved performance over RGrad for the Gaussian real case and especially the CDP 2D case when both algorithms use the constant stepsize $\alpha=0.2$. This suggests that truncation is helpful in avoiding overshooting to allow medium large constant stepsizes.

\begin{figure}[ht!]
\centering
	\includegraphics[width=0.4 \textwidth]{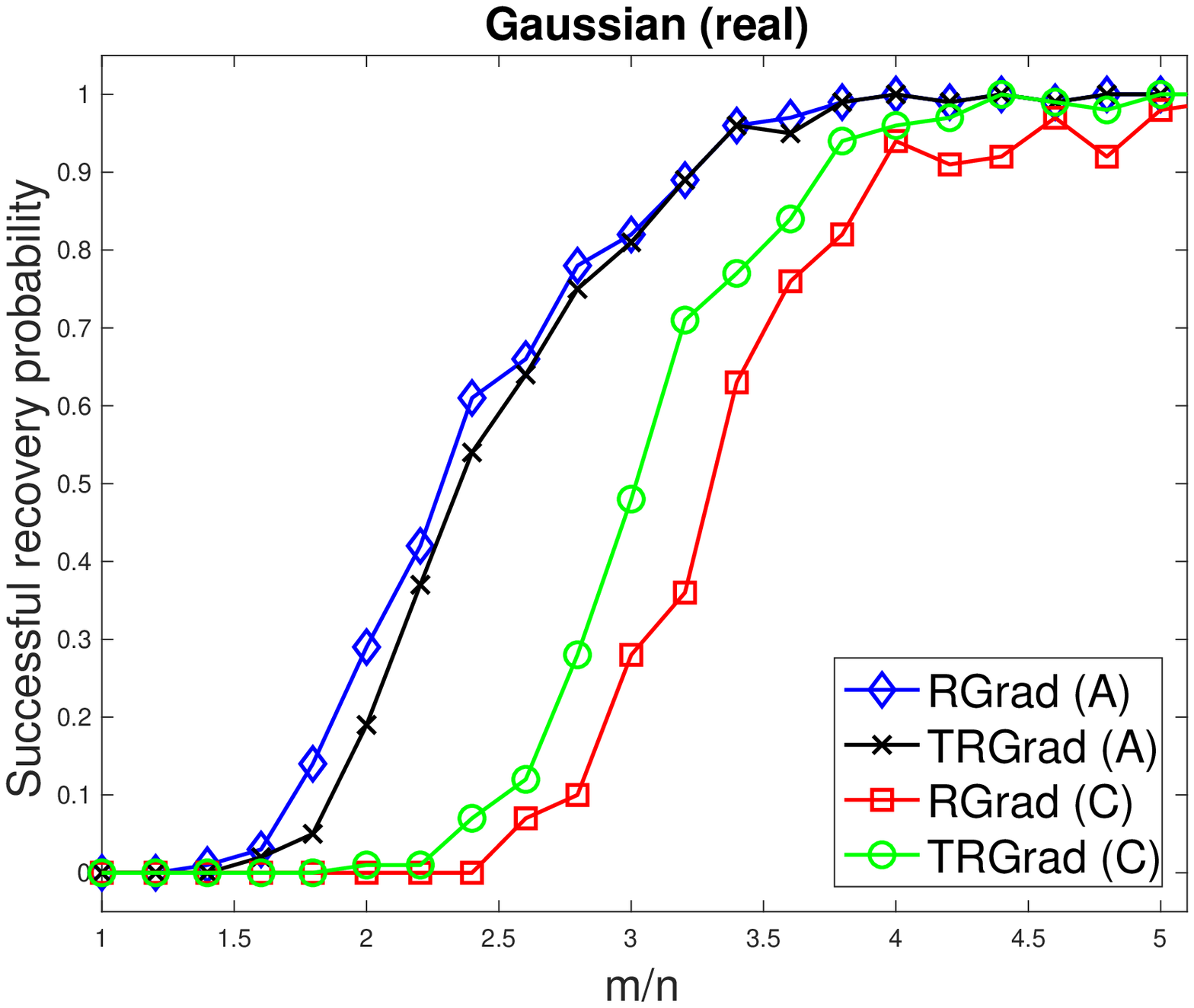}\hspace{0.5cm}
	\includegraphics[width=0.4 \textwidth]{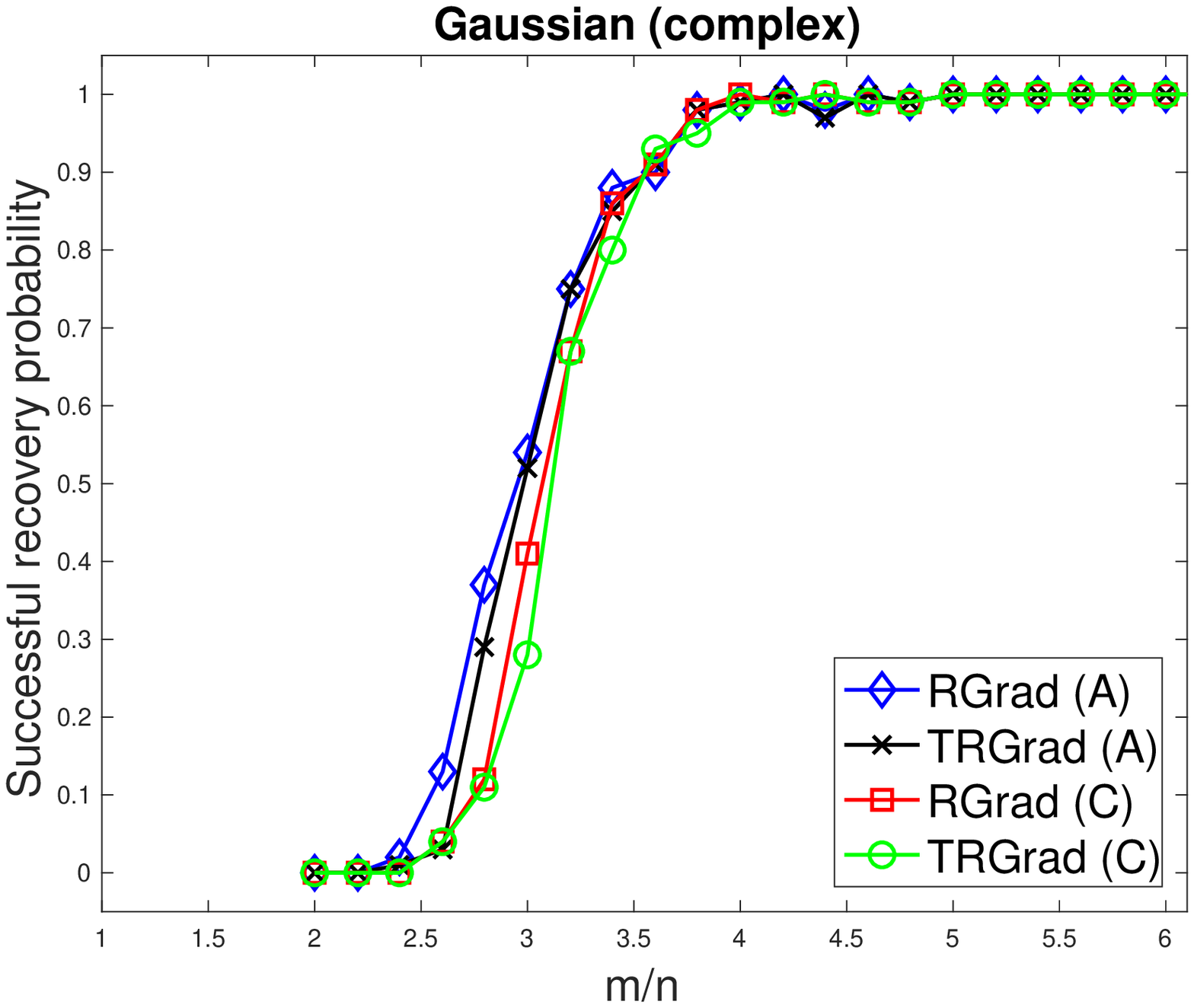}\\
	\vspace{0.3cm}
	\includegraphics[width=0.4 \textwidth]{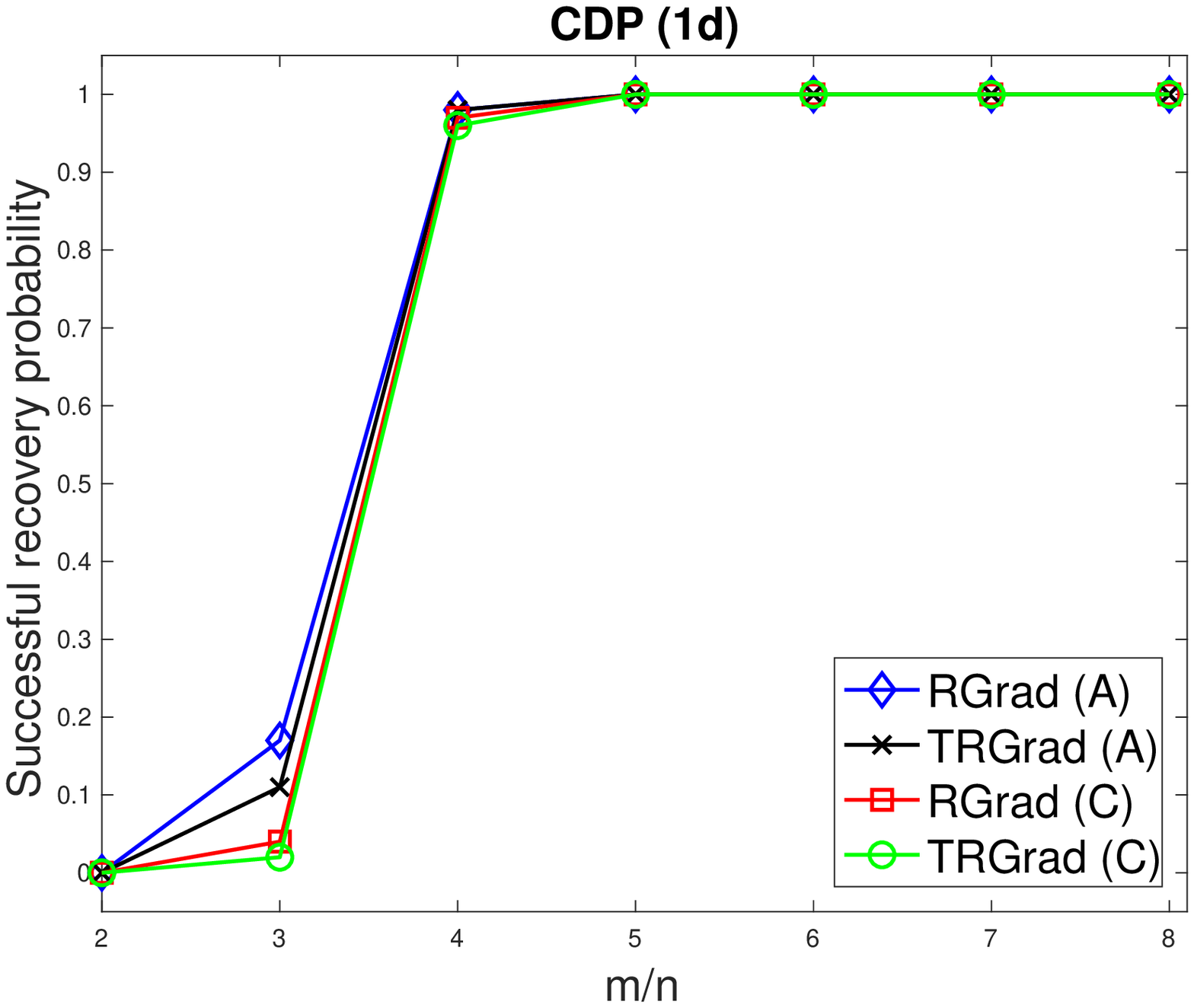}\hspace{0.5cm}
	\includegraphics[width=0.4 \textwidth]{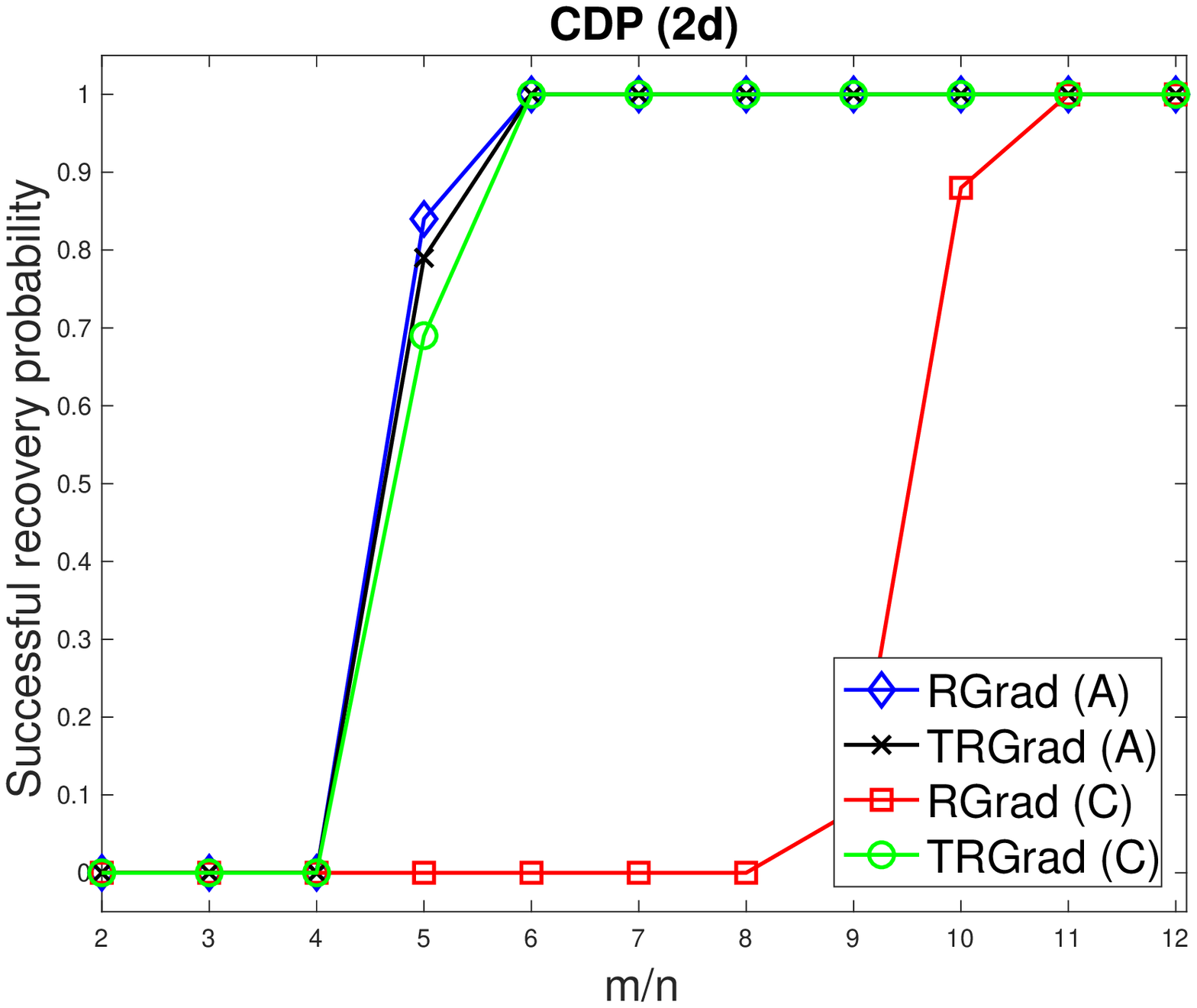}
\caption{Comparison of RGrad and TGRad for different measurement models.}
\label{fig1}
\end{figure}

\begin{figure}[ht!]
\centering
	\includegraphics[width=0.4 \textwidth]{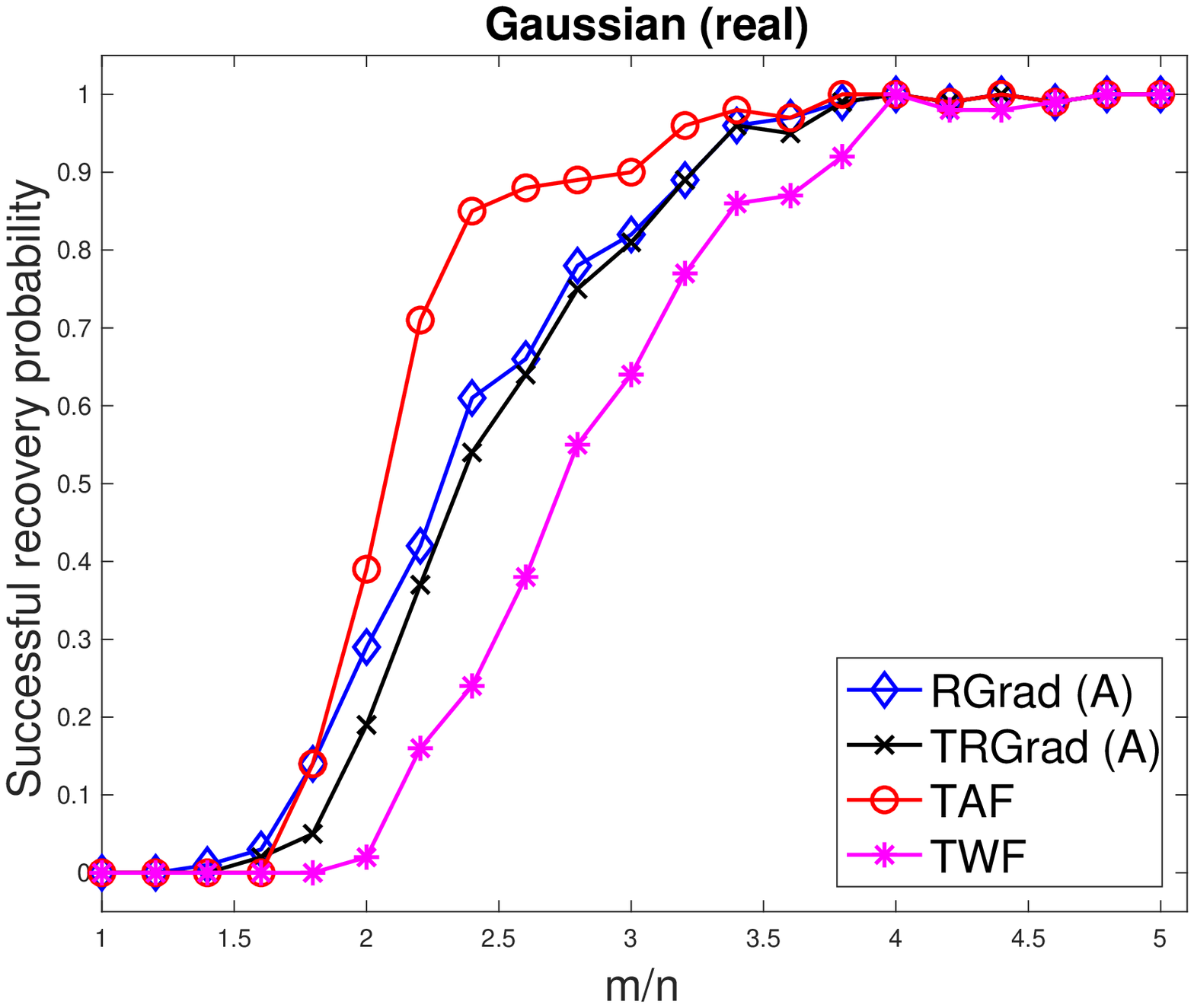}\hspace{0.5cm}
	\includegraphics[width=0.4 \textwidth]{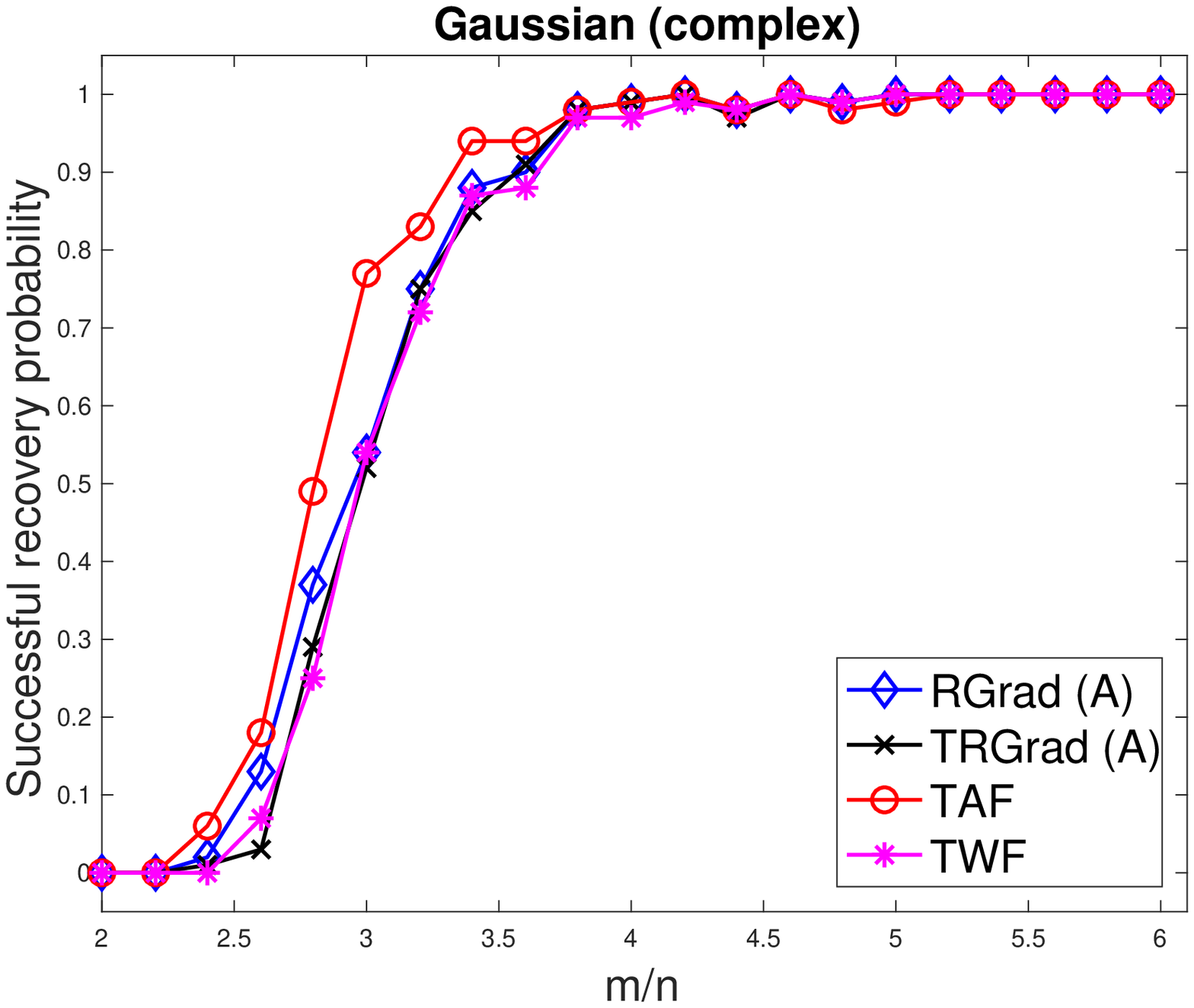}\\
	\vspace{0.3cm}
	\includegraphics[width=0.4 \textwidth]{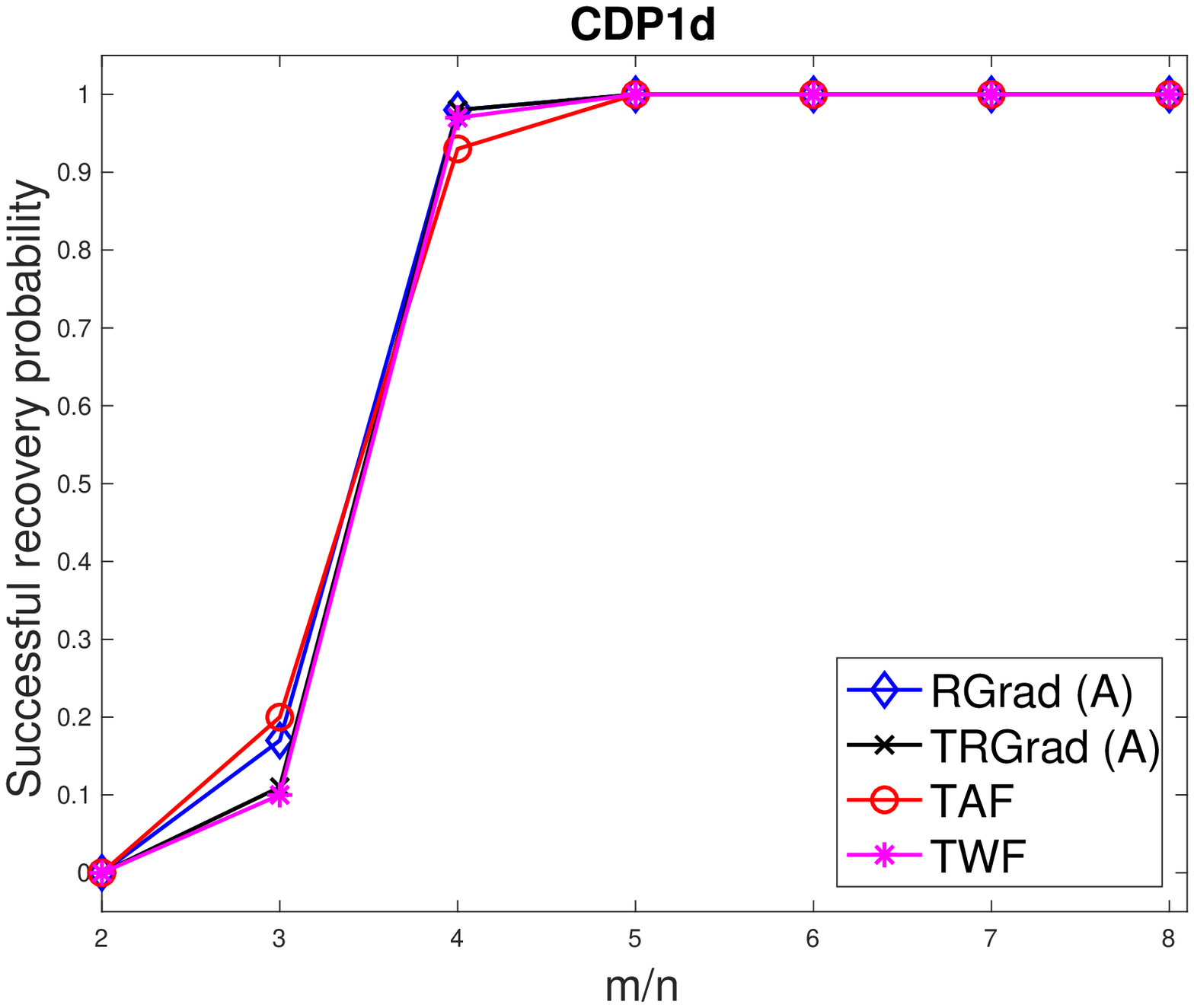}\hspace{0.5cm}
	\includegraphics[width=0.4 \textwidth]{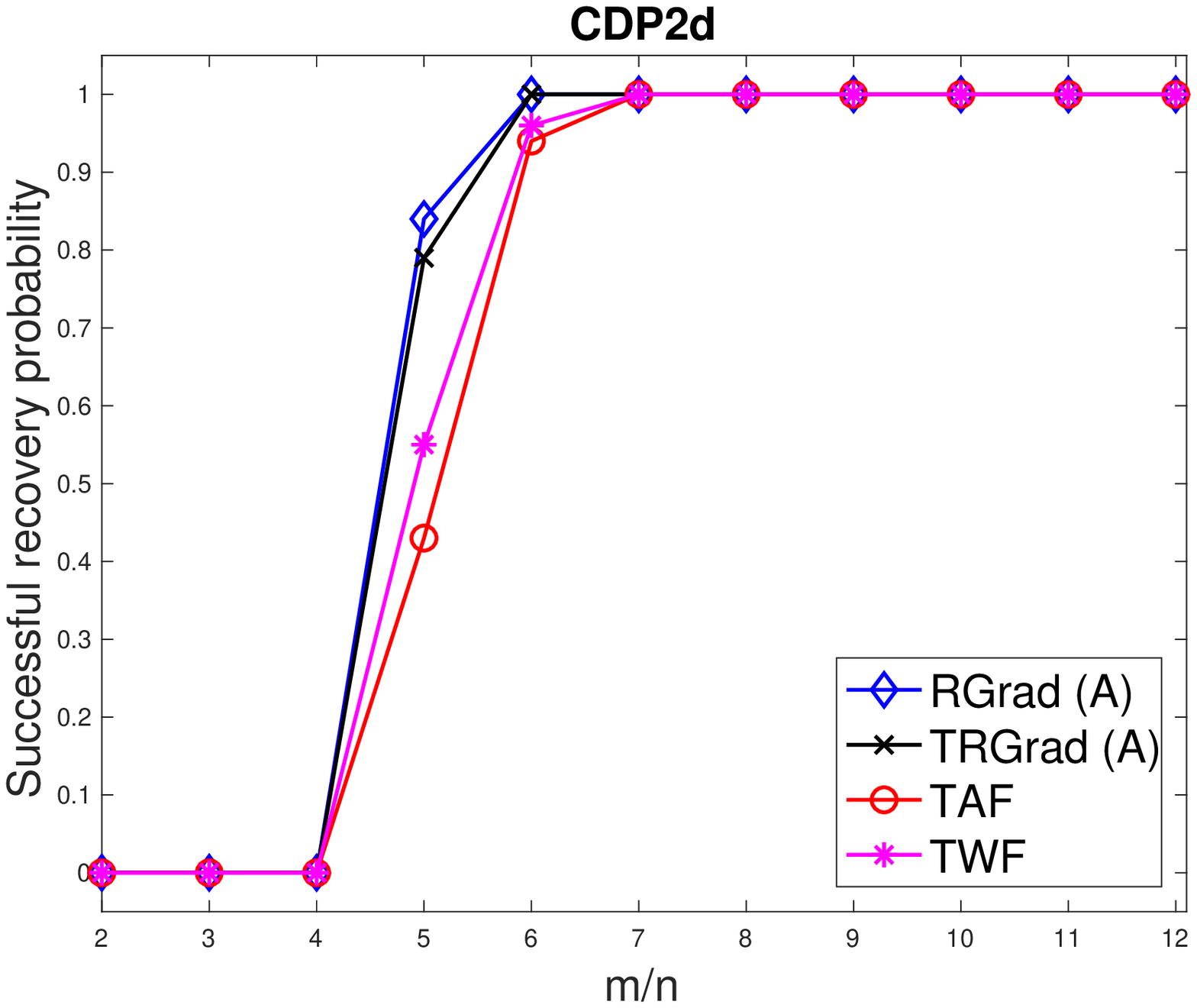}
\caption{Comparison of RGrad and TGRad for different measurement models.}
\label{fig2}
\end{figure}

Next  we compare recovery performance of RGrad, TRGrad, TAF and TWF; see Figure~\ref{fig2}. For clarity, only results for RGrad and TRGrad with adaptive steepest stepsize are presented. The figure shows that in the small oversampling region TAF has higher phase transition for the Gaussian case while RGrad and TRGrad have higher phase transition for the CDP 2D case. 

%%%%%%
\subsection{Computational time and stability }
\begin{figure}[ht!]
\centering
	\includegraphics[width=0.45\textwidth]{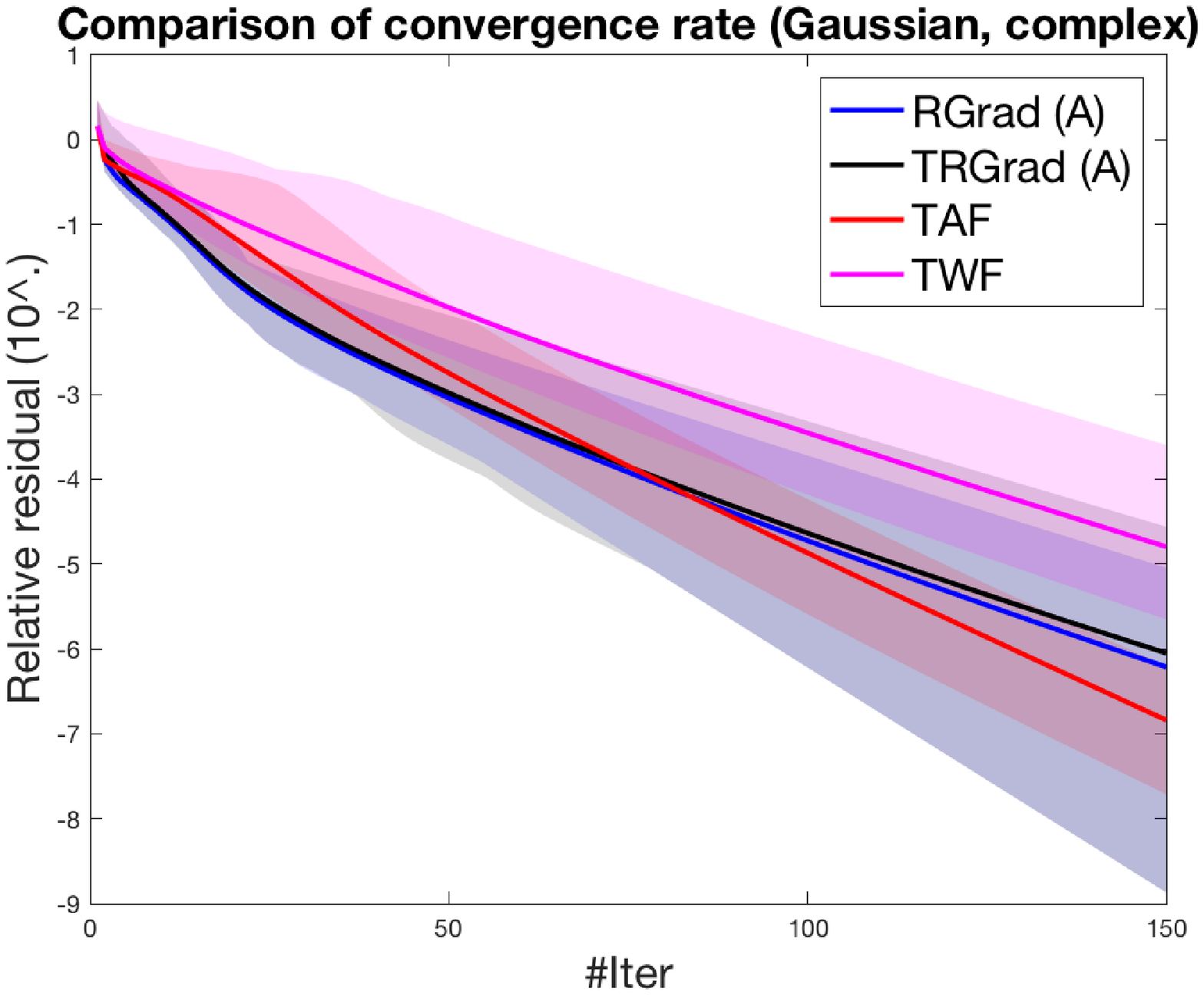}%\hspace{0.5cm}
	\includegraphics[width=0.45\textwidth]{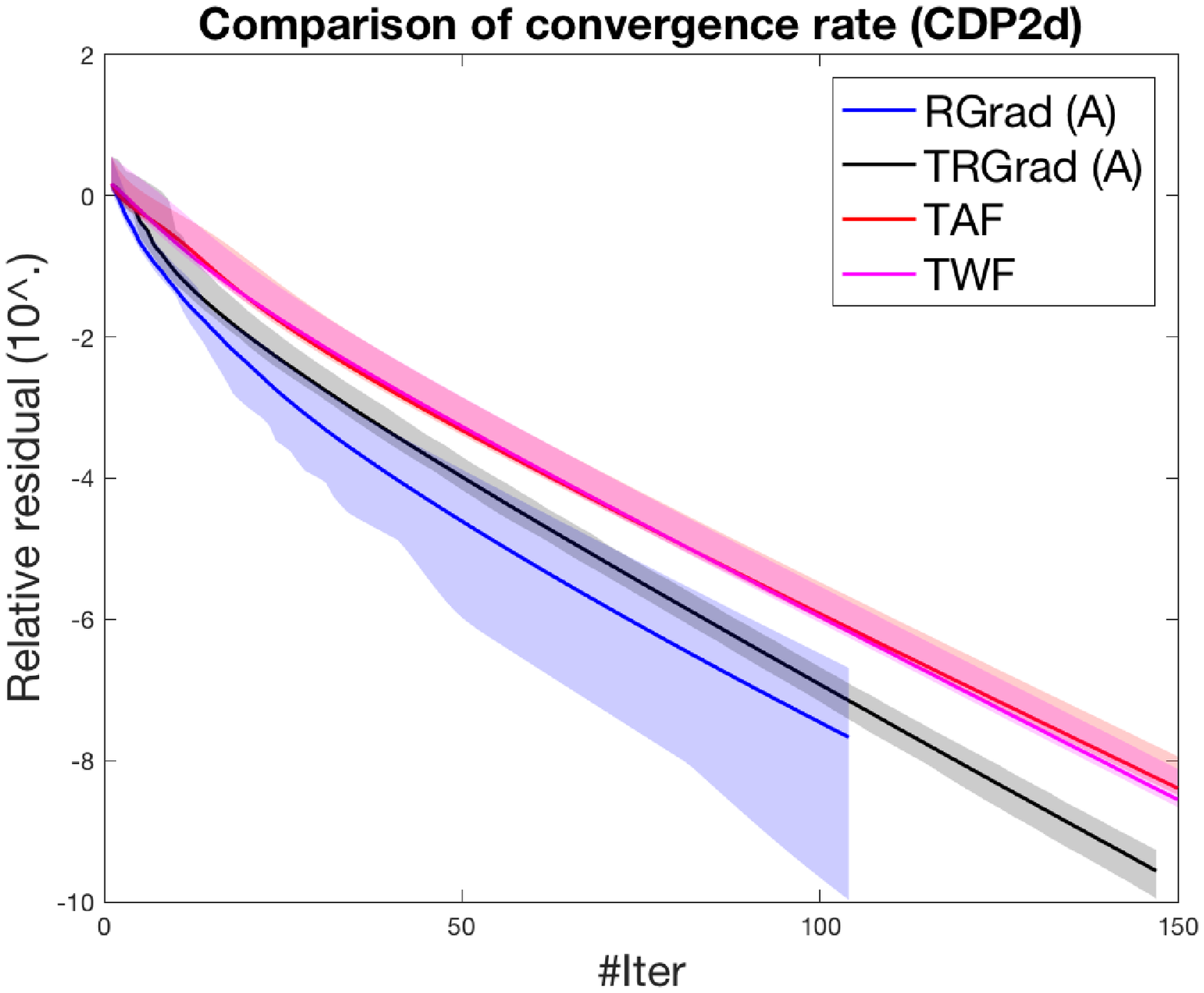}
\caption{Range and average of the relative residuals over $100$ random tests.}
\label{fig3}
\end{figure}

The dominant per iteration computational costs of all the four  test algorithms  lie in  the two matrix-vector products involving $\BA$ and $\BA^\top$. To investigate their computational efficiency, we consider the average convergence rates of the algorithms. Tests are conducted for the complex Gaussian case with $n/m=6$ and the CDP 2D case with $n/m=8$. The range and average of the relative residuals measured by $\||\BA\bx_k|^2-\by\|/\|\by\|$ over $100$ random simulations against the number of iterations are presented in Figure~\ref{fig3}. For the complex Gaussian case, TAF exhibits an overall superior performance while RGrad and TRGrad have a faster convergence rate for the CDP 2D case.

\begin{figure}[ht!]
\centering
	\includegraphics[width=0.4\textwidth]{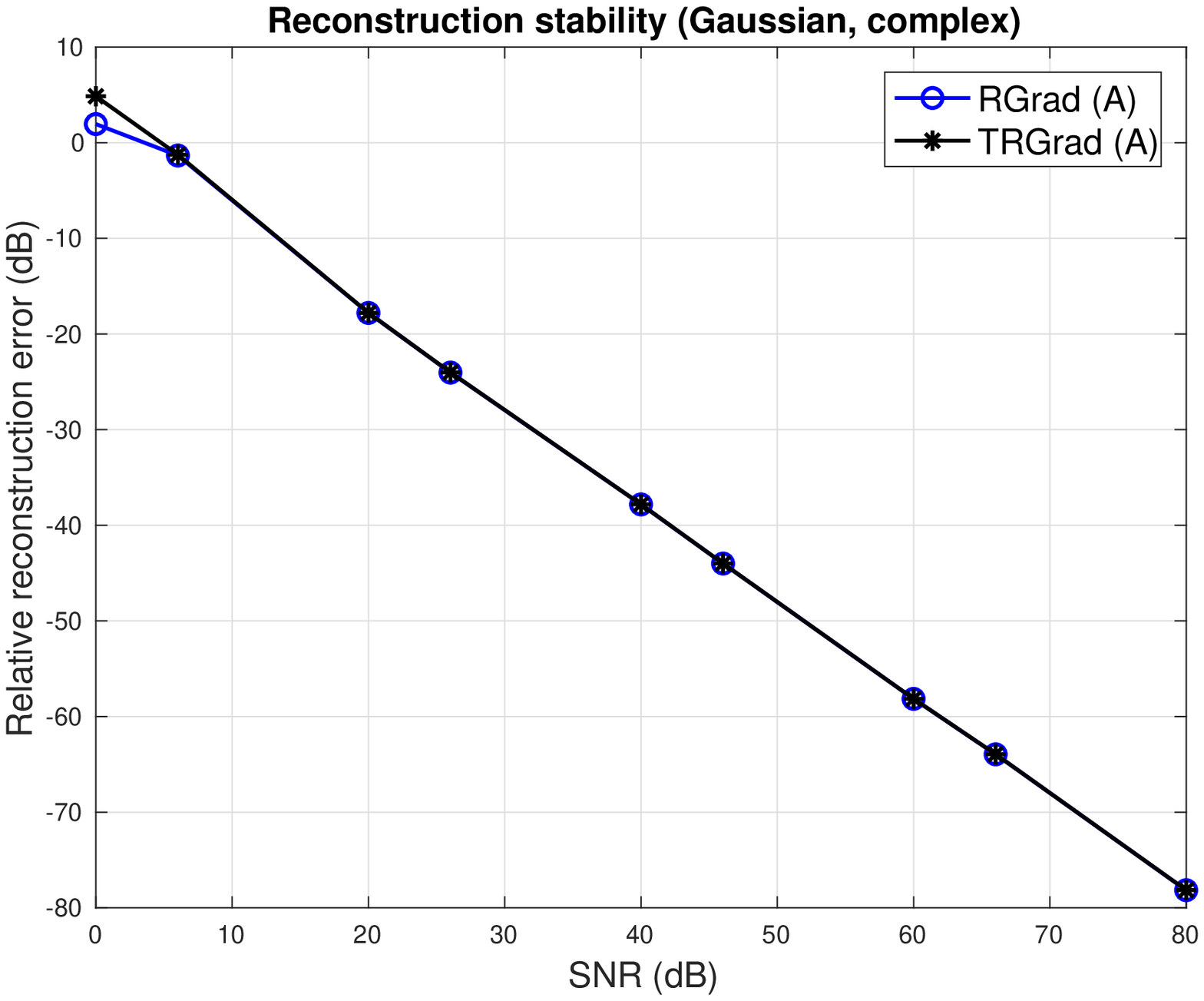}\hspace{0.5cm}
	\includegraphics[width=0.4\textwidth]{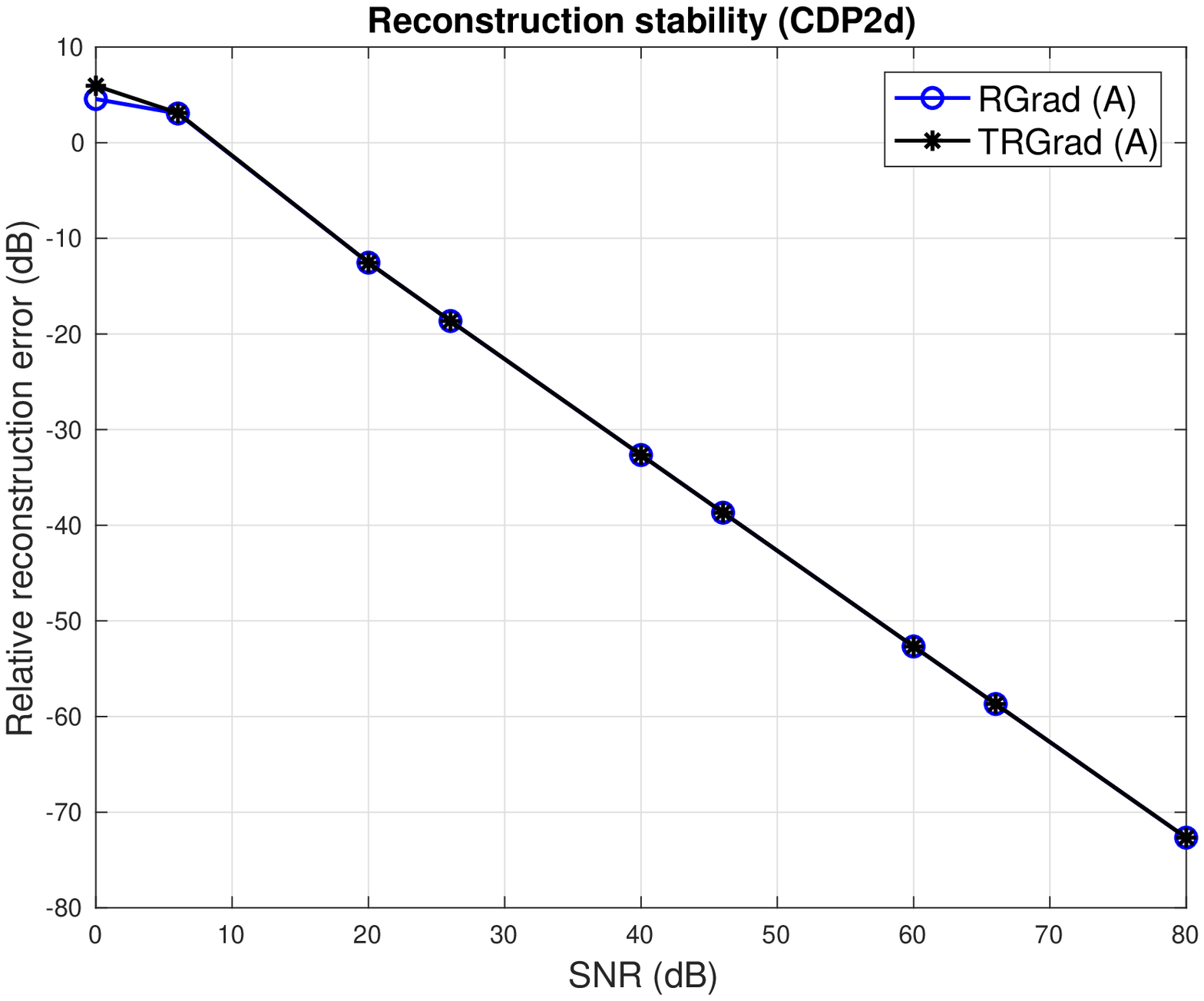}
\caption{Performance of RGrad and TGRad under different SNR.}
\label{fig4}
\end{figure}

We demonstrate the performance of RGrad and TRGrad under additive noise by conducting tests with the measurement measure $\by$ corrupted by 
\begin{align*}
\be = \sigma\cdot \|\by\|\cdot\frac{\bw}{\|\bw\|},
\end{align*}
where $\bw$ is a standard Gaussian random vector and $\sigma$ is the noise level.  As above, tests are conducted for the complex Gaussian and CDP 2D cases  with $9$ different values of $\sigma$, corresponding to $9$ different equispaced signal-to-noise ratios (SNRs). The average relative reconstruction error in dB plotted against the SNR is presented in Figure~\ref{fig4}.
The desirable linear scaling between the noise levels and the relative reconstruction
errors can be observed from the figure  for both RGrad and TRGrad.
\section{Proofs of main theorems}\label{sec:proofs1}
Recall that $\A_{\bz}$ is a truncated linear operator defined in \eqref{eq:Az}. Let $\BZ=\bz\bz^\top$ be a rank-$1$ and positive semidefinite matrix. The tangent space of the embedded manifold of rank-$1$ and positive semidefinite matrices at $\BZ$, denoted $T_{\bz}$, is given by (see also \eqref{eq:tangent})
\begin{align*}
T_{\bz}=\{\bz\bw^\top+\bw\bz^\top~|~\bw\in\R^n\}.
\end{align*}
The proof the main theorem relies on the local well conditioned property of $\A_{\bz}$ when being restricted onto the tangent space $T_{\bz}$, and the local weak correlation property of $\A_{\bz}$ between its restrictions onto $T_{\bz}$ and the complementary of $T_{\bz}$. %between the restriction of $\A_{\bz}$ onto $T_{\bz}$ and the restriction of $\A_{\bz}$
%onto the complementary of $T_{\bz}$.

\begin{theorem}[Restricted Well Conditioned]\label{thm:well_conditioned}
Let 
$ \vartheta_1:=\mean{|\xi|^4\onesub{|\xi|\leq \gamma}}-\mean{|\xi|^2\onesub{|\xi|\leq \gamma}}$ and $ \vartheta_2:=\mean{|\xi|^2\onesub{|\xi|\leq \gamma}}$ for $\xi$ being a standard normal distribution.
With probability exceeding $1-e^{-\Omega(m)}$, 
\begin{align*}
\btlow\ln\BW\rn^2_F\leq\frac{1}{m}\ln\A_{\bz}\lb\BW\rb\rn^2\leq \btup\ln\BW\rn_F^2
\end{align*}
holds uniformly { for all $\bz\bz^\top$ obeying $\|\bz\bz^\top-\bx\bx^\top\|_F\leq\frac{1}{13}\|\BX\|_F$} and all $\BW\in T_{\bz}$ provided $m\gtrsim n$. Here $$\btlow=2\min\lcb \vartheta_1, \vartheta_2-\rho_1-\rho_2\rcb\quad\mbox{and}\quad\btup=\max\lcb \vartheta_1+ \vartheta_2+\rho_1,2\lb \vartheta_2+\rho_1\rb\rcb$$ with $\rho_1=10\tau_z^3e^{-0.49\tau_z^2}+o(1)$ and $\rho_2=6\tau_z^2\tau_he^{-0.64\tau_h^2}+4\tau_z^2\tau_xe^{-0.39\tau_x^2}+o(1)$.
\end{theorem}

\begin{theorem}[Restricted Weak Correlation]\label{thm:weak_correlation}
Let $\varepsilon_0$ be the numerical constant defined in \eqref{eq:basin_size2}
Then with probability exceeding $1-e^{-\Omega(m)}$,
\begin{align*}
\frac{1}{m}\ln\P_{T_{\bz}}\A_{\bz}^\top\A_{\bz}\lb\I-\P_{T_{\bz}}\rb\lb \bz\bz^\top-\bx\bx^\top\rb\rn_F\leq\sqrt{\rho_4\btup}\ln\bz\bz^\top-\bx\bx^\top\rn_F
\end{align*}
holds uniformly for all $\bz\bz^\top$ obeying $\|\bz\bz^\top-\bx\bx^\top\|_F\leq \varepsilon_0\|\BX\|_F$ provided $m\gtrsim n$,  where $\rho_4=\rho_3+6\tau_h\tau_{h,z}^2e^{-0.64\tau_h^2}+o(1)$.
\end{theorem}
The proofs for Theorems~\ref{thm:well_conditioned} and \ref{thm:weak_correlation}
will be deferred to Section~\ref{sec:proofs2}. It is worth noting that in Theorem~\ref{thm:well_conditioned} we have 
\begin{align*}
\btlow\leq \vartheta_1+\vartheta_2-\rho_1-\rho_2\leq \vartheta_1+\vartheta_2+\rho_1\leq\btup.
\end{align*}

The following corollary follows immediately from Theorem~\ref{thm:well_conditioned}.
\begin{corollary}\label{thm:truncated_iso}
Let $\alpha>0$ be an absolute constant. With probability exceeding $1-e^{-\Omega(m)}$,
\begin{align*}
\ln\P_{T_{\bz}}-\frac{\alpha}{m}\P_{T_{\bz}}\A_{\bz}^\top\A_{\bz}\P_{T_{\bz}}\rn\leq \max\{\lab 1-\alpha\btlow\rab,\lab 1-\alpha\btup\rab\}\numberthis\label{eq:cor_truncated_iso}
\end{align*}
holds uniformly { for all $\bz\bz^\top$ obeying $\|\bz\bz^\top-\bx\bx^\top\|_F\leq\frac{1}{13}\|\BX\|_F$}. Moreover, $\max\{\lab 1-\alpha\btlow\rab,\lab 1-\alpha\btup\rab\}<1$ when $\alpha<\frac{2}{\btup}$ and the minimum is achieved at $\alpha=\frac{2}{\btlow+\btup}$ with the value given by $\frac{\btup-\btlow}{\btup+\btlow}$.
\end{corollary}
\begin{proof}
By Theorem~\ref{thm:well_conditioned}, one can easily see that
\begin{align*}
\lab\ln\BW\rn_F^2-\frac{\alpha}{m}\ln\A_{\bz}\lb\BW\rb\rn^2\rab\leq\max\{\lab 1-\alpha\btlow\rab,\lab 1-\alpha\btup\rab\}\ln\BW\rn^2_F
\end{align*}
holds for all for all $\bz\bz^\top$ obeying $\|\bz\bz^\top-\bx\bx^\top\|_F\leq\frac{1}{13}\|\BX\|_F$ and $\BW\in T_{\bz}$.
Thus, \eqref{eq:cor_truncated_iso} can be established by noting that 
$$
\ln\P_{T_{\bz}}-\frac{\alpha}{m}\P_{T_{\bz}}\A_{\bz}^\top\A_{\bz}\P_{T_{\bz}}\rn=\max_{\BW\in T_{\bz},~\ln\BW\rn_F=1}\lab\ln\BW\rn_F^2-\frac{\alpha}{m}\ln \A_{\bz}\lb\BW\rb\rn^2\rab.
$$
The rest of the claims can be verified easily. 
\end{proof}

Letting $\BW_l = \BZ_l+\alpha_l\P_{T_l}(\BG_l)$, we can further establish the following lemma.
\begin{lemma}\label{lem:W_and_Z}
If $\ln\BW_l-\BX\rn_F\leq \mu\ln\BZ_l-\BX\rn_F$ and $\ln\BZ_l-\BX\rn_F\leq\varepsilon_0\ln\BX\rn_F$, then 
\begin{align*}
\ln\BZ_{l+1}-\BX\rn_F\leq \mu\sqrt{1+16\mu^2\varepsilon_0^2}\ln\BZ_l-\BX\rn_F.
\end{align*}
\end{lemma}

%%%%
\begin{proof}
Since $\BZ_{k+1}=\T_1(\BW_l)$ is the closest positive semidefinite rank-$1$ matrix to $\BW_l$, it follows that 
\begin{align*}
\|\BZ_{l+1}-\BX\|_F\leq \|\BZ_{l+1}-\BW_l\|_F+\|\BW_l-\BX\|_F\leq 2\|\BW_l-\BX\|_F.\numberthis\label{eq:M_and_Z_01}
\end{align*}
Recall that $T_{l+1}$ is the tangent space of rank-$1$ positive semidefinite matrix at $\BZ_{l+1}$. One has $\BZ_{l+1}=\P_{T_{l+1}}(\BW_l)$, and hence
\begin{align*}
\|\BZ_{l+1}-\BX\|_F^2& = \|\P_{T_{l+1}}(\BW_l)-\BX\|_F^2\\
&=\|\P_{T_{l+1}}(\BW_l-\BX)\|_F^2+\|(\I-\P_{T_{l+1}})(\BX)\|_F^2\\
&\leq \|\BW_l-\BX\|_F^2+\frac{\|\BZ_{l+1}-\BX\|_F^4}{\|\BX\|_F^2}\\
&\leq  \|\BW_l-\BX\|_F^2+\frac{16\|\BW_l-\BX\|_F^4}{\|\BX\|_F^2}\\
&\leq \mu^2\|\BZ_l-\BX\|_F^2+\frac{16\mu^4\|\BZ_l-\BX\|_F^4}{\|\BX\|_F^2}\\
&\leq \mu^2(1+16\mu^2\varepsilon_0^2)\|\BZ_l-\BX\|_F^2,
\end{align*}
where the third line follows from \cite[Lemma~4.1]{WCCL:SIMAX:16}\footnote{Lemma~4.1 in \cite{WCCL:SIMAX:16} was established for non-symmetric matrices and the corresponding tangent spaces, but the result can be easily extended to the symmetric case.}, and the fourth line follows from \eqref{eq:M_and_Z_01}.
\end{proof}
%===========
Now we are in position to prove Theorems~\ref{thm:main} and \ref{thm:rgd_sd_local}.
\begin{proof}[Proof of Theorem~\ref{thm:main}]
Assume $\|\BZ_l-\BX\|_F\le \varepsilon_0\|\BX\|_F$ which can be proved by mathematical induction for all $l\geq 0$ once we show that $\|\BZ_{l+1}-\BX\|_F\leq \nu_g\|\BZ_l-\BX\|_F$. {Noting that 
$
\A_l^\top(\by) = \A_l^\top\A_l(\BX),
$}
we have 
$
\BW_l =\BZ_l+\frac{\alpha_l}{m}\P_{T_l}\A_l^\top\A_l(\BX-\BZ_l),
$
so
\begin{align*}
\|\BW_l-\BX\|_F &= \| (\BZ_l-\BX)-\frac{\alpha_l}{m}\P_{T_l}\A_l^\top\A_l(\BZ_l-\BX)\|_F\\
& \leq \|(\P_{T_l}-\frac{\alpha_l}{m}\P_{T_l}\A_l^\top\A_l\P_{T_l})(\BZ_l-\BX)\|_F+\|(\I-\P_{T_l})\BX\|_F\\
&\quad+\frac{\alpha_l}{m}\|\P_{T_l}\A_l^\top\A_l(\I-\P_{T_l})(\BZ_l-\BX)\|_F.\numberthis\label{eq:iter_relation}
\end{align*}
By Corollary~\ref{thm:truncated_iso}, we have
\begin{align*}
\|(\P_{T_l}-\frac{\alpha_l}{m}\P_{T_l}\A_l^\top\A_l\P_{T_l})(\BZ_l-\BX)\|_F\leq \max\{|1-\alpha_l\btlow|,|1-\alpha_l\btup|\}\|\BZ_l-\BX\|_F.
\end{align*}
By Theorem~\ref{thm:weak_correlation}, we have 
\begin{align*}
\frac{\alpha_l}{m}\|\P_{T_l}\A_l^\top\A_l(\I-\P_{T_l})(\BZ_l-\BX)\|_F\leq \alpha_l\sqrt{\rho_4\btup}\ln\BZ_l-\BX\rn_F.
\end{align*}
By \cite[Lemma~4.1]{WCCL:SIMAX:16}, we have 
\begin{align*}
\|(\I-\P_{T_l})\BX\|_F\leq \frac{\|\BZ_l-\BX\|_F^2}{\|\BX\|_F}\leq \varepsilon_0\|\BZ_l-\BX\|_F.
\end{align*}
Substituting the above three bounds into \eqref{eq:iter_relation} yields 
\begin{align*}
\|\BW_l-\BX\|_F\leq\mu\|\BZ_l-\BX\|_F,
\end{align*}
where $\mu=\max\{|1-\alpha_l\btlow|,|1-\alpha\btup|\}+\alpha_l\sqrt{\rho_4\btup}+\varepsilon_0$.

It follows from Lemma~\ref{lem:W_and_Z} that 
\begin{align*}
\ln\BZ_{l+1}-\BX\rn_F\leq \mu\sqrt{1+16\mu^2\varepsilon_0^2}\ln\BZ_l-\BX\rn_F.
\end{align*}
Define $\nu_g=\mu\sqrt{1+16\mu^2\varepsilon_0^2}$. It is easy to see that $\nu_g<1$ as long as 
\begin{align*}
\mu\leq \frac{1}{\sqrt{1+16\varepsilon_0^2}},
\end{align*}
which in turn requires 
\begin{align*}
\max\{|1-\alpha_l\btlow|,|1-\alpha\btup|\}+\alpha_l\sqrt{\rho_4\btup}\leq \epup,\numberthis\label{eq:bound4Rw}
\end{align*}
where $\epup=\lb 1-\varepsilon_0\sqrt{1+16\varepsilon_0}\rb/\sqrt{1+16\varepsilon_0}$.

Noting that \begin{align*}
\max\{|1-\alpha_l\btlow|,|1-\alpha_l\btup|\}=
\begin{cases}
\alpha_l\btup-1 &\mbox{if } \alpha_l \in \lsb \frac{2}{\btlow+\btup},\frac{2}{\btup}\rb\\
1-\alpha_l\btlow & \mbox{if }\alpha_l\in\lb0, \frac{2}{\btlow+\btup}\rsb,
\end{cases}
\end{align*}
a simple calculation shows that \eqref{eq:bound4Rw} can be satisfied if 
\begin{align*}
\alpha_l\in\lsb\frac{1-\epup}{\btlow-\sqrt{\rho_4\btup}},\frac{1+\epup}{\btup+\sqrt{\rho_4\btup}}\rsb
\end{align*}
conditioned on 
\begin{align*}
\btlow>\sqrt{\rho_4\btup}\quad\mbox{and}\quad(1-\epup)\btup+2\sqrt{\rho_4\btup}\leq (1+\epup)\btlow,
\end{align*}
which concludes the proof.
\end{proof}
%==============
\begin{proof}[Proof of Theorem~\ref{thm:rgd_sd_local}]
By Theorem~\ref{thm:well_conditioned}, we know that 
\begin{align*}
\frac{1}{\btup}\leq\alpha_l\leq\frac{1}{\btlow}.
\end{align*}
Thus, in order to show the linear convergence of TRGrad with the steepest descent stepsize, it only requires to verify that (see \eqref{eq:stepsize_range})
\begin{align*}
\frac{1-\epup}{\btlow-\sqrt{\rho_4\btup}}\leq\frac{1}{\btup}\leq \frac{1}{\btlow}\leq \frac{1+\epup}{\btup+\sqrt{\rho_4\btup}},
\end{align*}
which can be satisfied under the assumption.
\end{proof}
\section{Proofs for Section~\ref{sec:proofs1}}\label{sec:proofs2}
%%%%%
\subsection{Auxiliary events and properties}\label{subsec:auxiliary_events}
In this section we introduce a few auxiliary events to facilitate the analysis and present the properties of the events. 
The set of auxiliary are summarized as follows: 
\begin{align*}
&\ME_2^k(\bx) = \lcb |\ba_k^\top\bx|\leq 0.9\tau_x\|\bx\|\rcb,\\
&\ME_3^k(\bx) = \lcb |\ba_k^\top\bx|\leq 1.1\tau_x\|\bx\|\rcb,\\
&\ME^k_3(\bz)=\lcb\lab y_k-|\ba_k^\top\bz|^2\rab\leq 1.15\tau_h\|\bh\| \lb |\ba_k^\top\bz|+\sqrt{y_k} \rb\rcb,\\
&\ME^k_4(\bz)=\lcb\lab y_k-|\ba_k^\top\bz|^2\rab\leq 3\tau_h\|\bh\| \lb |\ba_k^\top\bz|+\sqrt{y_k} \rb\rcb,\\
&\ME^k_5(\bz)=\lcb| \ba_k^\top\bh|\leq 1.15\tau_h\|\bh\|\rcb,\\
&\ME^k_6(\bz)=\lcb| \ba_k^\top\bh|\leq 3\tau_h\|\bh\|\rcb,
\end{align*}
where in the last four events $\bh=\bz-\bx$. The following two lemmas establish the connection between the auxiliary events and the events that determine the truncation rules in Algorithm~\ref{alg:tRGrad}.

%%%
\begin{lemma}\label{prop:events_prop_x}
With probability at least $1-e^{-\Omega(m)}$, we have 
\begin{align*}
\ME^k_2(\bx)\subset \ME^k_1(\bx)\subset \ME^k_3(\bx)
\end{align*}
provided $m\gtrsim n$.
\end{lemma}
\begin{proof}
When $\delta\in(0,0.5)$, it follows from \cite{CSV:CPAM:13} that
\begin{align*}
(1-\delta)\ln\bz\rn^2\leq \frac{1}{m}\sum_{k=1}^n|\ba_k^\top\bz|^2\leq (1+\delta)\ln\bz\rn^2\numberthis\label{eq:event_04}
\end{align*}
holds for all $\bz\in\R^n$ with probability $1-2e^{-m\epsilon^2/2}$ provided $m\geq 20\delta^{-2}n$, where $\delta=4(\epsilon^2+\epsilon)$. Since  $\frac{1}{m}\ln\by\rn_1=\frac{1}{m}\sum_{k=1}^n|\ba_k^\top\bx|^2$, the proof is complete by choosing $\delta$ properly.
\end{proof}

%%%
\begin{lemma}\label{prop:events_prop_z}
With probability at least $1-e^{-\Omega(m)}$, 
\begin{align*}
\ME_5^k(\bz)\subset \ME_3^k(\bz),\quad\ME_6^k(\bz)\subset \ME_4^k(\bz),\quad\mbox{and}\quad
\ME_3^k(\bz)\subset \ME_2^k(\bz){\subset\ME_4^k(\bz)}.\numberthis\label{eq:event_01}
\end{align*}
hold for all $\bz$ and $\bx$ satisfying $\ln\bh\rn\leq\frac{1}{11}\ln\bz\rn$ provided $m\gtrsim n$. Under the same condition, one has 
\begin{align*}
\ME_1^k(\bz)\cap\ME_1^k(\bx)\cap\ME_j^k(\bz)\subset\lcb |\ba_k^\top\bh|\leq \tau_{h,z}\|\bz\|\rcb,\quad j=3,4, \numberthis\label{eq:event_02}
\end{align*}
where
$\tau_{h,z}=\tau_z+\lb 0.3\tau_h\lb\tau_z+1.2\tau_x\rb+\tau_z^2 \rb^{1/2}$.
\end{lemma}

\begin{proof}
The first two claims of \eqref{eq:event_01} can be verified directly. For example, when $|\ba_k^\top\bh|\leq 1.15\tau_h\|\bh\|$, we have 
\begin{align*}
\lab y_k-|\ba_k^\top\bz|^2\rab&=\lab |\ba_k^\top\bx|^2-|\ba_k^\top\bz|^2\rab\leq \lb|\ba_k^\top\bx|+|\ba_k^\top\bz|\rb|\ba_k^\top\bh|\\
&\leq 1.15\tau_h\|\bh\|\lb |\ba_k^\top\bz|+\sqrt{y_k}\rb.
\end{align*}
The third claim of \eqref{eq:event_01} follows immediately from \cite[Eq.~(5.9)]{CC:CPAM:17}, which states that with probability $1-e^{-\Omega(m)}$,
\begin{align*}
1.15\ln\bh\rn\ln\bz\rn\leq \frac{1}{m}\ln\A\lb\bz\bz^\top-\bx\bx^\top\rb\rn_1\leq 3\ln\bh\rn\ln\bz\rn
\end{align*}
holds for all $\bz$ and $\bx$ satisfying $\ln\bh\rn\leq\frac{1}{11}\ln\bz\rn$ provided $m\gtrsim n$.

By Lemma~\ref{prop:events_prop_z}, it suffices to show that 
\begin{align*}
{\ME_1^k(\bz)\cap\ME_3^k(\bx)\cap\ME_j^k(\bz)\subset\lcb |\ba_k^\top\bh|\leq \tau_{h,z}\ln\bz\rn\rcb,\quad j=3,4}.
\end{align*}
We only need to check the case when $j=4$. Note that 
\begin{align*}
\lab y_k-|\ba_k^\top\bz|^2\rab & = |\ba_k^\top(\bz-\bx)||\ba_k^\top(\bz-\bx)|\\
& = |\ba_k^\top(2\bz-\bh)||\ba_k^\top\bh|\\
&\geq |\ba_k^\top\bh|^2-2|\ba_k^\top\bh||\ba_k^\top\bz|\\
&=\lb|\ba_k^\top\bh|-|\ba_k^\top\bz|\rb^2-|\ba_k^\top\bz|^2,
\end{align*}
where in the second line we have used the substitution $\bh=\bz-\bx$. Thus, when $\|\bh\|\leq\frac{1}{11}\|\bz\|$, for any outcome from $\ME_1^k(\bz)\cap\ME_3^k(\bx)\cap\ME_4^k(\bz)$, we have 
\begin{align*}
|\ba_k^\top\bh|&\leq |\ba_k^\top\bz|+\sqrt{\lab y_k-|\ba_k^\top\bz|^2\rab+|\ba_k^\top\bz|^2}\\
&\leq |\ba_k^\top\bz|+\sqrt{3\tau_h\|\bh\| \lb |\ba_k^\top\bz|+\sqrt{y_k} \rb+|\ba_k^\top\bz|^2}\\
&\leq \tau_z\|\bz\|+\sqrt{ 0.3\tau_h\|\bz\| \lb \tau_z\|\bz\|+ 1.1\tau_x(\|\bz\| +\|\bh\|)\rb+\tau_z^2\|\bz\|^2}\\
&\leq \tau_{h,z}\|\bz\|,
\end{align*}
which completes the proof.
\end{proof}

%%%%%%%%
\subsection{Spectral norm of random matrices with truncation}\label{subsec:spectral_norm}
The following technical lemma  which might be of independent interest will be used repeatedly in our analysis. It
provides a uniform bound for a set of random matrices parameterized by an arbitrary vector.
\begin{lemma}\label{lem:key1}
Fix $\gamma \geq 1$ and let $\epsilon\in(0,1)$ be a sufficiently small constant. With probability at least $1-e^{-\Omega(m\epsilon^2)}$,
 \begin{align*}
\ln\frac{1}{m}\sum_{k=1}^m\ba_k\ba_k^\top\dsone_{\lcb|\ba_k^\top\bz|>\gamma\ln\bz\rn\rcb}\rn\leq 5\gamma e^{-0.49\gamma^2}+\epsilon\numberthis\label{eq:trunc_2}
 \end{align*}
 holds uniformly for all $\ln\bz\rn\neq 0$ provided $m\gtrsim \epsilon^{-2}\log\epsilon^{-1}\cdot n$.
 \end{lemma}
 %%%
 \begin{proof}
 By homogeneity, we only need to establish \eqref{eq:trunc_2} for the case where $\|\bz\|=1$. Let $\N_{1/4}$ be the $1/4$-net of $S^{n-1}$. By \cite[Lemma~5.4]{Ver2011rand}, it suffices to show that 
 \begin{align*}
  \frac{1}{m}\sum_{k=1}^m|\ba_k^\top\bw|^2\dsone_{\lcb|\ba_k^\top\bz|>\gamma\rcb}\leq 2.5\gamma e^{-0.49\gamma^2}+\epsilon
 \end{align*}
 holds simultaneously for all $\bw\in\N_{1/4}$ and $\bz\in S^{n-1}$. We first have 
 \begin{align*}
 &\frac{1}{m}\sum_{k=1}^m|\ba_k^\top\bw|^2\onesub{|\ba_k^\top\bz|>\gamma}\\
 &=\frac{1}{m}\sum_{k=1}^m|\ba_k^\top\bw|^2\onesub{|\ba_k^\top\bw|\leq\zeta}\onesub{|\ba_k^\top\bz|>\gamma}+
 \frac{1}{m}\sum_{k=1}^m|\ba_k^\top\bw|^2\onesub{|\ba_k^\top\bw|>\zeta}\onesub{|\ba_k^\top\bz|>\gamma}\\
& \leq \frac{\zeta^2}{m}\sum_{k=1}^m\onesub{|\ba_k^\top\bz|>\gamma}+
 \frac{1}{m}\sum_{k=1}^m|\ba_k^\top\bw|^2\onesub{|\ba_k^\top\bw|>\zeta}\\
 &:=\TM_1+\TM_2,
 \end{align*}
 where $\zeta>0$ is a numerical constant that can be chosen flexibly. Then we only need to bound $\TM_1$ uniformly for all $\bz\in S^{n-1}$ and bound $\TM_2$ uniformly for all $\bw\in\N_{1/4}$.
 
%%%
\paragraph{Upper bound of $\TM_1$ over $\bz\in S^{n-1}$}
To this end, define an auxiliary function $f(x)$ for $x\geq 0$ as 
  \begin{align*}
  f(x) = \begin{cases}
  1 & \mbox{if }x\in [\gamma,\infty),\\
   \frac{1}{\delta\gamma}x+\lb 1-\frac{1}{\delta}\rb& \mbox{if }x\in[(1-\delta)\gamma,\gamma]),\\
   0&\mbox{if }x\in[0,(1-\delta)\gamma),
  \end{cases}
  \end{align*}
  where  $\delta>0$ is a very small constant to be determined later. Hence $f(x)$
  is continuous on $[0,\infty)$ and $\dsone_{\lcb|\ba_k^\top\bz|>\gamma\rcb}\leq f\lb|\ba_k^\top\bz|\rb\leq \dsone_{\lcb|\ba_k^\top\bz|>(1-\delta)\gamma\rcb}$. Moreover, 
  it can be easily verified that $f\lb\sqrt{\tau}\rb$ is Lipschitz continuous with the Lipschitz constant bounded by $\frac{1}{2\delta(1-\delta)\gamma^2}$. It follows that 
  \begin{align*}
  \frac{1}{m}\sum_{k=1}^m\onesub{|\ba_k^\top\bz|\geq \gamma}\leq \frac{1}{m}\sum_{k=1}^mf\lb |\ba_k^\top\bz|\rb,\numberthis\label{eq:trunc_03}
  \end{align*}
  so it suffices to bound $\frac{1}{m}\sum_{k=1}^mf\lb |\ba_k^\top\bz|\rb$ uniformly 
  for all $\|\bz\|=1$. Note that  
  \begin{align*}
  \sup_{p\geq 1}p^{-1}\lb\mean{\lb f\lb |\ba_k^\top\bz|\rb\rb^p}\rb^{1/p}&
  \leq \sup_{p\geq 1}p^{-1}\lb\mean{\lb\onesub{|\ba_k^\top\bz|>(1-\delta)\gamma}\rb^p}\rb^{1/p}\\
 &=\sup_{p\geq 1}p^{-1}\lb \prob{|\ba_k^\top\bz|>(1-\delta)\gamma}\rb^{1/p}\\
  &\leq \sup_{p\geq 1}p^{-1}\lb \sqrt{\frac{2}{\pi}}\frac{1}{(1-\delta)\gamma}e^{-\frac{(1-\delta)^2\gamma^2}{2}}\rb^{1/p}\\
  &\leq \sup_{p\geq 1}p^{-1}e^{-\frac{(1-\delta)^2\gamma^2}{2p}}\\
  &\leq \frac{1}{(1-\delta)^2\gamma^2},
  \end{align*}
  where in the fourth line we assume $(1-\delta)\gamma\geq \sqrt{2/\pi}$ which holds for $\gamma\geq 1$ and $\delta\leq 0.2$. Thus, $f\lb|\ba_k^\top\bz|\rb$ is sub-exponential with the sub-exponential norm $\|\cdot\|_{\psi_1}$ obeying 
  \begin{align*}
  \ln f\lb|\ba_k^\top\bz|\rb\rn_{\psi_1}\lesssim \frac{1}{(1-\delta)^2\gamma^2},
  \end{align*}
  and so is $f\lb|\ba_k^\top\bz|\rb-\mean{f\lb|\ba_k^\top\bz|\rb}$.  Thus, by the Bernstein's inequality (see for example \cite{Ver2011rand}), we have 
  \begin{align*}
\frac{1}{m}\sum_{k=1}^mf\lb |\ba_k^\top\bz|\rb \leq 
  \mean{f\lb |\ba_k^\top\bz|\rb}+\frac{\epsilon}{(1-\delta)^2\gamma^2}
  \leq  \sqrt{\frac{2}{\pi}}\frac{1}{(1-\delta)\gamma}e^{-\frac{(1-\delta)^2\gamma^2}{2}}+\frac{\epsilon}{(1-\delta)^2\gamma^2}\numberthis\label{eq:trunc_02}
  \end{align*}
  holds with probability at least $1-e^{-\Omega(m\epsilon^2)}$ for $\epsilon\in(0,1)$.
  Let $\N_\epsilon$ be $\epsilon$-net of $S^{n-1}$. Applying the union bound implies \eqref{eq:trunc_02} holds for all $\bz\in\N_\epsilon$ with probability exceeding $1-e^{-\Omega(m\epsilon^2)}$ provided $m\gtrsim \epsilon^{-2}\log\epsilon^{-1}\cdot n$. 
  
For any $\bz\in S^{n-1}$, let $\bz_0\in\N_\epsilon$ be a vector satisfying $\ln\bz-\bz_0\rn\leq \epsilon$. Then we have
\begin{align*}
&\lab\frac{1}{m}\sum_{k=1}^mf\lb |\ba_k^\top\bz|\rb-\frac{1}{m}\sum_{k=1}^mf\lb |\ba_k^\top\bz_0|\rb\rab\\
&\leq \frac{1}{m}\sum_{k=1}^m\lab f\lb |\ba_k^\top\bz|\rb-f\lb |\ba_k^\top\bz_0|\rb\rab\\
&=\frac{1}{m}\sum_{k=1}^m\lab f\lb\sqrt{ |\ba_k^\top\bz|^2}\rb-f\lb \sqrt{|\ba_k^\top\bz_0|^2}\rb\rab\\
&\leq \frac{1}{2\delta(1-\delta)\gamma^2}\cdot\frac{1}{m}\sum_{k=1}^m\lab |\ba_k^\top\bz|^2-|\ba_k^\top\bz_0|^2\rab\\
&\leq \frac{1}{2\delta(1-\delta)\gamma^2}\cdot\sqrt{\frac{1}{m}\sum_{k=1}^m|\ba_k^\top(\bz+\bz_0)|^2}\sqrt{\frac{1}{m}\sum_{k=1}^m|\ba_k^\top(\bz-\bz_0)|^2}\\
&\leq \frac{1+\epsilon}{2\delta(1-\delta)\gamma^2}\cdot\|\bz+\bz_0\|\|\bz-\bz_0\|\\
&\leq \frac{(1+\epsilon)\epsilon}{2\delta(1-\delta)\gamma^2},\numberthis\label{eq:trunc_04}
\end{align*}
where in the fourth line we use the fact that $f\lb\tau\rb$ is $\frac{1}{2\delta(1-\delta)\gamma^2}$-Lipschitz, and the sixth line holds uniformly for all $\bz$ and $\bz_0$ with probability at least $1-e^{-\Omega(m\epsilon^2)}$ provided $m\gtrsim \epsilon^{-2}n$. To sum up, when $\delta\in(0,1/2)$ and $\epsilon\in (0,1)$, we can bound $\TM_1$ uniformly for all $\bz\in S^{n-1}$ as 
\begin{align*}
\TM_1\leq \sqrt{\frac{2}{\pi}}\frac{\zeta^2}{(1-\delta)\gamma}e^{-\frac{(1-\delta)^2\gamma^2}{2}}+\frac{3\zeta^2\epsilon}{\delta(1-\delta)\gamma^2}
\end{align*}
by combining \eqref{eq:trunc_03}, \eqref{eq:trunc_02} and \eqref{eq:trunc_04} together.

\paragraph{Upper bound $\TM_2$ over $\bw\in \N_{1/4}$}
It is clear that $|\ba_k^\top\bw|^2\dsone_{\lcb|\ba_k^\top\bw|>\zeta\rcb}$ is sub-exponential since $|\ba_k^\top\bw|^2\dsone_{\lcb|\ba_k^\top\bw|>\zeta\rcb}\leq |\ba_k^\top\bw|^2$ and $|\ba_k^\top\bw|^2$ is standard Chi-square and sub-exponential. Therefore, applying the Bernstein inequality yields that 
\begin{align*}
\frac{1}{m}\sum_{k=1}^m|\ba_k^\top\bw|^2\dsone_{\lcb|\ba_k^\top\bw|>\zeta\rcb}&\leq \mean{|\ba_k^\top\bw|^2\dsone_{\lcb|\ba_k^\top\bw|>\zeta\rcb}}+\epsilon\\
&\leq \sqrt{\frac{2}{\pi}}\lb\zeta+\frac{1}{\zeta}\rb e^{-\frac{\zeta^2}{2}}+\epsilon
\end{align*}
holds with probability at least $1-e^{-\Omega(m\epsilon^2)}$ for $\epsilon\in (0,1)$ being sufficiently small. Taking a union bound over $\N_{1/4}$ yields that 
\begin{align*}
\TM_2\leq \sqrt{\frac{2}{\pi}}\lb\zeta+\frac{1}{\zeta}\rb e^{-\frac{\zeta^2}{2}}+\epsilon
\end{align*}
holds for all $\bw\in\N_{1/4}$ with probability exceeding $1-e^{-\Omega(m\epsilon^2)}$ provided $m\gtrsim \epsilon^{-2} n$. \\

Finally, taking $\zeta=\gamma$ and choosing {$\delta=0.01$}  completes the proof of the lemma. 
 \end{proof}
 
 %%%
 The following lemma is  a the result from  \cite{truncated_wf_supp}.
To keep the presentation self-contained, we provide a slightly different proof here based on Lemma~\ref{lem:key1}. In particular,  the dependence of the upper bound on the parameters will be made explicit.

\begin{lemma}\label{lem:key2}
Fix $\gamma\geq 2$ and let $\epsilon\in(0,1)$ be a sufficiently small constant. With probability at least $1-e^{-\Omega(m\epsilon^2)}$, 
\begin{align*}
&\ln\frac{1}{m}\sum_{k=1}^m|\ba_k^\top\bz|^2\ba_k\ba_k^\top\onesub{|\ba_k^\top\bz|\leq\gamma\|\bz\|}-\lb \vartheta_1\bz\bz^\top+ \vartheta_2\|\bz\|^2\BI\rb\rn\\
&\leq \lb10\gamma^3e^{-0.49\gamma^2}+\frac{10\gamma^2}{\epsilon}e^{-0.49\epsilon^{-2}}+\gamma^4\epsilon\rb\|\bz\|^2
\end{align*}
holds for all $\bz\in\R^n$ provided $m\gtrsim \epsilon^{-2}\log\epsilon^{-1}\cdot n$, where $ \vartheta_1:=\mean{|\xi|^4\onesub{|\xi|\leq \gamma}}-\mean{|\xi|^2\onesub{|\xi|\leq \gamma}}$ and $ \vartheta_2:=\mean{|\xi|^2\onesub{|\xi|\leq \gamma}}$ with $\xi$ being a standard normal distribution.
\end{lemma}
%%%
\begin{proof}
For fixed $\bz\in S^{n-1}$, the proof is standard and can also be found for example in \cite{CLS:TIT:15}. By \cite[Lemma~5.4]{Ver2011rand}, it suffices to bound 
\begin{align*}
\lab \frac{1}{m}\sum_{k=1}^m |\ba_k^\top\bz|^2|\ba_k^\top\bw|^2\onesub{|\ba_k^\top\bz|\leq\gamma}-
\lb \vartheta_1|\bz^\top\bw|^2+ \vartheta_2\|\bw\|^2\rb\rab
\end{align*}
for all $\bw\in\N_{1/4}$, where $\N_{1/4}$ is a $1/4$-net of $S^{n-1}$. To this end, a simple calculation yields that 
\begin{align*}
&\lab \frac{1}{m}\sum_{k=1}^m |\ba_k^\top\bz|^2|\ba_k^\top\bw|^2\onesub{|\ba_k^\top\bz|\leq\gamma}-
\lb \vartheta_1|\bz^\top\bw|^2+ \vartheta_2\|\bw\|^2\rb\rab\\
&\leq (\bw^\top\bz)^2\lab\frac{1}{m}\sum_{k=1}^m(\ba_k^\top\bz)^4\onesub{|\ba_k^\top\bz|\leq\gamma}-(\vartheta_1+\vartheta_2)\rab\\
&+2|\bw^\top\bz|\lab\frac{1}{m}\sum_{k=1}^m(\ba_k^\top\bz)^3(\ba_k^\top\tilde{\bz})\onesub{|\ba_k^\top\bz|\leq\gamma}\rab\\
&+\|\tilde{\bz}\|^2\lab\frac{1}{m}\sum_{k=1}^m(\ba_k^\top\bz)^2\onesub{|\ba_k^\top\bz|\leq\gamma}-\vartheta_2\rab\\
&+\lab\frac{1}{m}\sum_{k=1}^m(\ba_k^\top\bz)^2\lb(\ba_k^\top\tilde{\bz})^2-\|\tilde{\bz}\|^2\rb\onesub{|\ba_k^\top\bz|\leq\gamma}\rab.
\end{align*}
where we have used the decomposition $\bw=(\bw^\top\bz)\bz+\tilde{\bz}$ with $\tilde{\bz}\perp\bz$ and $1=\|\bw\|^2=(\bw^\top\bz)^2+\|\tilde{\bz}\|^2$. After applying the Hoeffding inequality to the first three terms and applying the Bernstein inequality to the last term, we can easily see that for fixed $\bw\in S^{n-1}$,
\begin{align*}
\lab \frac{1}{m}\sum_{k=1}^m |\ba_k^\top\bz|^2|\ba_k^\top\bw|^2\onesub{|\ba_k^\top\bz|\leq\gamma}-
\lb \vartheta_1|\bz^\top\bw|^2+ \vartheta_2\|\bw\|^2\rb\rab\leq (\gamma^4+\gamma^3+\gamma^2)\epsilon
\end{align*}
holds with probability at least $1-e^{-\Omega(m\epsilon^2)}$ for $\epsilon\in(0,1)$ being sufficiently small. Taking a uniform  bound over all $\bw\in\N_{1/4}$ implies that 
\begin{align*}
\ln\frac{1}{m}\sum_{k=1}^m|\ba_k^\top\bz|^2\ba_k\ba_k^\top\onesub{|\ba_k^\top\bz|\leq\gamma}-\lb \vartheta_1\bz\bz^\top+ \vartheta_2\BI\rb\rn\leq\lb\gamma^4+\gamma^3+\gamma^2\rb\epsilon\numberthis\label{eq:trunc_05}
\end{align*}
holds for fixed $\bz\in S^{n-1}$ with probability at least  $1-e^{-\Omega(m\epsilon^2)}$ provided $m\gtrsim \epsilon^{-2}n$.

To establish a uniform bound for all $\bz\in S^{n-1}$, first note that \eqref{eq:trunc_05} holds for all $\bz\in\N_{\epsilon^2}$ provided $m\gtrsim \epsilon^{-2}\log\epsilon^{-1}\cdot n$, where $\N_{\epsilon^2}$ denotes a $\epsilon^2$-net of $S^{n-1}$. For any $\bz\in S^{n-1}$, let $\bz_0$ be a vector in $\N_{\epsilon^2}$ such that $\ln\bz-\bz_0\rn\leq \epsilon^2$. Then it follows that
\begin{align*}
&\ln \frac{1}{m}\sum_{k=1}^m|\ba_k^\top\bz|^2\ba_k\ba_k^\top\onesub{|\ba_k^\top\bz|\leq\gamma}-\frac{1}{m}\sum_{k=1}^m|\ba_k^\top\bz_0|^2\ba_k\ba_k^\top\onesub{|\ba_k^\top\bz_0|\leq\gamma}\rn\\
&\leq \ln\frac{1}{m}\sum_{k=1}^m\ba_k\ba_k^\top\lb |\ba_k^\top\bz|^2-|\ba_k^\top\bz_0|^2\rb\onesub{|\ba_k^\top\bz|\leq \gamma}\onesub{|\ba_k^\top\bz_0|\leq \gamma}\rn\\
&+\ln\frac{1}{m}\sum_{k=1}^m\ba_k\ba_k^\top(\ba_k^\top\bz)^2\onesub{|\ba_k^\top\bz|\leq \gamma}\onesub{|\ba_k^\top\bz_0|> \gamma}\rn+
\ln\frac{1}{m}\sum_{k=1}^m\ba_k\ba_k^\top(\ba_k^\top\bz_0)^2\onesub{|\ba_k^\top\bz|> \gamma}\onesub{|\ba_k^\top\bz_0|\leq\gamma}\rn\\
&\leq \ln\frac{1}{m}\sum_{k=1}^m\ba_k\ba_k^\top\lb |\ba_k^\top\bz|^2-|\ba_k^\top\bz_0|^2\rb\onesub{|\ba_k^\top\bz|\leq \gamma}\onesub{|\ba_k^\top\bz_0|\leq \gamma}\onesub{|\ba_k^\top(\bz-\bz_0)|\leq\epsilon}\rn\\
&+\ln\frac{1}{m}\sum_{k=1}^m\ba_k\ba_k^\top\lb |\ba_k^\top\bz|^2-|\ba_k^\top\bz_0|^2\rb\onesub{|\ba_k^\top\bz|\leq \gamma}\onesub{|\ba_k^\top\bz_0|\leq \gamma}\onesub{|\ba_k^\top(\bz-\bz_0)|>\epsilon}\rn\\
&+\ln\frac{1}{m}\sum_{k=1}^m\ba_k\ba_k^\top(\ba_k^\top\bz)^2\onesub{|\ba_k^\top\bz|\leq \gamma}\onesub{|\ba_k^\top\bz_0|> \gamma}\rn+
\ln\frac{1}{m}\sum_{k=1}^m\ba_k\ba_k^\top(\ba_k^\top\bz_0)^2\onesub{|\ba_k^\top\bz|> \gamma}\onesub{|\ba_k^\top\bz_0|\leq\gamma}\rn\\
&\leq \ln\frac{1}{m}\sum_{k=1}^m\ba_k\ba_k^\top\lb |\ba_k^\top\bz|^2-|\ba_k^\top\bz_0|^2\rb\onesub{|\ba_k^\top\bz|\leq \gamma}\onesub{|\ba_k^\top\bz_0|\leq \gamma}\onesub{|\ba_k^\top(\bz-\bz_0)|\leq\epsilon}\rn\\
&+\ln\frac{1}{m}\sum_{k=1}^m\ba_k\ba_k^\top\lb |\ba_k^\top\bz|^2-|\ba_k^\top\bz_0|^2\rb\onesub{|\ba_k^\top\bz|\leq \gamma}\onesub{|\ba_k^\top\bz_0|\leq \gamma}\onesub{|\ba_k^\top(\bz-\bz_0)|>\epsilon^{-1}\|\bz-\bz_0\|}\rn\\
&+\ln\frac{1}{m}\sum_{k=1}^m\ba_k\ba_k^\top(\ba_k^\top\bz)^2\onesub{|\ba_k^\top\bz|\leq \gamma}\onesub{|\ba_k^\top\bz_0|> \gamma}\rn+
\ln\frac{1}{m}\sum_{k=1}^m\ba_k\ba_k^\top(\ba_k^\top\bz_0)^2\onesub{|\ba_k^\top\bz|> \gamma}\onesub{|\ba_k^\top\bz_0|\leq\gamma}\rn\\
&\leq 2\gamma\epsilon\ln \frac{1}{m}\sum_{k=1}^m\ba_k\ba_k^\top\rn+2\gamma^2\ln\sum_{k=1}^m\ba_k\ba_k^\top\onesub{|\ba_k^\top(\bz-\bz_0)|>\epsilon^{-1}\|\bz-\bz_0\|}\rn\\
&\quad+ \gamma^2\ln\frac{1}{m}\sum_{k=1}^m\ba_k\ba_k^\top\onesub{|\ba_k^\top\bz_0|>\gamma}\rn+ \gamma^2\ln\frac{1}{m}\sum_{k=1}^m\ba_k\ba_k^\top\onesub{|\ba_k^\top\bz| >\gamma}\rn\\
&\leq 10\gamma^3e^{-0.49\gamma^2}+\frac{10\gamma^2}{\epsilon}e^{-0.49\epsilon^{-2}}+\gamma^2\epsilon\numberthis\label{eq:trunc_06}
\end{align*}
where in the third inequality we have used the fact $\|\bz-\bz_0\|\leq\epsilon^2$, and the last inequality holds with probability $1-e^{-\Omega(m\epsilon^2)}$ provided $m\gtrsim \epsilon^{-2}\log\epsilon^{-1}\cdot n$; see \eqref{eq:event_04} and Lemma~\ref{lem:key1}.

By combining \eqref{eq:trunc_05} and \eqref{eq:trunc_06} together,  for any $\bz\in S^{n-1}$, we have 
\begin{align*}
&\ln\frac{1}{m}\sum_{k=1}^m|\ba_k^\top\bz|^2\ba_k\ba_k^\top\onesub{|\ba_k^\top\bz|\leq\gamma}-\lb \vartheta_1\bz\bz^\top+ \vartheta_2\BI\rb\rn\\
&\leq \ln \frac{1}{m}\sum_{k=1}^m|\ba_k^\top\bz|^2\ba_k\ba_k^\top\onesub{|\ba_k^\top\bz|\leq\gamma}-\frac{1}{m}\sum_{k=1}^m|\ba_k^\top\bz_0|^2\ba_k\ba_k^\top\onesub{|\ba_k^\top\bz_0|\leq\gamma}\rn\\
&+\ln\frac{1}{m}\sum_{k=1}^m|\ba_k^\top\bz_0|^2\ba_k\ba_k^\top\onesub{|\ba_k^\top\bz_0|\leq\gamma}-\lb \vartheta_1\bz_0\bz_0^\top+ \vartheta_2\BI\rb\rn\\
&+\ln\lb \vartheta_1\bz_0\bz_0^\top+ \vartheta_2\BI\rb-\lb \vartheta_1\bz\bz^\top+ \vartheta_2\BI\rb\rn\\
&\leq 10\gamma^3e^{-0.49\gamma^2}+\frac{{10}\gamma^2}{\epsilon}e^{-0.49\epsilon^{-2}}+\gamma^4\epsilon,
\end{align*}
which completes the proof.
\end{proof}
%%%%%%%%
\subsection{Proof of Theorem~\ref{thm:well_conditioned}}
The following lemma which relates $\|\bz-\bx\|$ to $\|\bz\bz^\top-\bx\bx^\top\|_F$ will be used later.

\begin{lemma}\label{lem:vector_bd_matrix}
For any $\bz\in\R^n$ and $\bx\in\R^n$ satisfying $\bz^\top\bx\geq 0$, we have 
\begin{align*}
\|\bz\bz^\top-\bx\bx^\top\|_F^2\geq\frac{4}{5}\|\bz-\bx\|^2\|\bx\|^2.
\end{align*}
\end{lemma}
\begin{proof}
Without loss of generality, assume $\|\bx\|=1$. Then, 
\begin{align*}
\|\bz\bz^\top-\bx\bx^\top\|_F^2-\frac{4}{5}\|\bz-\bx\|^2\|\bx\|^2& = \|\bz\|^4-\frac{4}{5}\|\bz\|^2-2(\bz^\top\bx)^2+\frac{8}{5}(\bz^\top\bx)+\frac{1}{5}\\
&=\|\bz\|^4-\frac{4}{5}\|\bz\|^2-2\ln\bz\rn^2t^2+\frac{8}{5}\ln\bz\rn t+\frac{1}{5}\numberthis\label{eq:local001}
\end{align*}
for since $0\leq t\leq 1$ due to $\ln\bx\rn=1$ and $\bz^\top\bx\geq 0$. Noting that \eqref{eq:local001} can be rewritten as 
\begin{align*}
\|\bz\|^4-\frac{4}{5}\|\bz\|^2-2\ln\bz\rn^2t^2+\frac{8}{5}\ln\bz\rn t+\frac{1}{5}= \|\bz\|^4-\frac{4}{5}\|\bz\|^2 -2\lb\|\bz\|t-\frac{2}{5}\rb^2+\frac{13}{25},
\end{align*}
hence for fixed $\bz$, 
\begin{align*}
\|\bz\bz^\top-\bx\bx^\top\|_F^2-\frac{4}{5}\|\bz-\bx\|^2\|\bx\|^2\geq 
\begin{cases}
\|\bz\|^4-\frac{4}{5}\|\bz\|^2+\frac{1}{5} & t = 0\\
\|\bz\|^4-\frac{14}{5}\|\bz\|^2+\frac{8}{5}\|\bz\|+\frac{1}{5} & t =1,
\end{cases}
\end{align*}
which are both greater than zero in both cases. 
\end{proof}
%%%
\begin{proof}[Proof of Theorem~\ref{thm:well_conditioned}]
Noting that for any $\BW = \bz\bw^\top+\bw\bz^\top$, there holds 
\begin{align*}
\|\BW\|_F^2=2\|\bz\|^2\|\bw\|^2+2|\bz^\top\bw|^2.
\end{align*}
\paragraph{Upper bound} Since $\dsone_{\ME_1^k(\bx)\cap \ME_1^k(\bz)\cap\ME_2^k(\bz)}\leq \dsone_{\ME_1^k(\bz)}$, one has 
\begin{align*}
\frac{1}{m}\|\A_{\bz}(\BW)\|^2& =\frac{1}{m}\sum_{k=1}^m\lab\langle\ba_k\ba_k^\top,\bz\bw^\top+\bw\bz^\top\rangle\rab^2\\
&\leq \frac{4}{m}\sum_{k=1}^m|\ba_k^\top\bz|^2|\ba_k^\top\bw|^2\dsone_{\ME_1^k(\bz)}\\
&=4\bw^\top\lb\frac{1}{m}\sum_{k=1}^m|\ba_k^\top\bz|^2\ba_k\ba_k^\top\onesub{|\ba_k^\top\bz|\leq\tau_z\|\bz\|}\rb\bw.
\end{align*}
{For fixed $\tau_z$, if we choose a sufficiently small $\epsilon$} in Lemma~\ref{lem:key2}, then with probability at least $1-e^{-\Omega(m)}$,
\begin{align*}
\ln \frac{1}{m}\sum_{k=1}^m|\ba_k^\top\bz|^2\ba_k\ba_k^\top\onesub{|\ba_k^\top\bz|\leq\tau_z\|\bz\|}-\lb \vartheta_1\bz\bz^\top+ \vartheta_2\|\bz\|^2\BI\rb\rn\leq \rho_1\|\bz\|^2\numberthis\label{eq:truncated_iso}
\end{align*}
holds for all $\bz\in\R^n$, where $ \vartheta_1:=\mean{|\xi|^4\onesub{|\xi|\leq \tau_z}}-\mean{|\xi|^2\onesub{|\xi|\leq \tau_z}}$ and $ \vartheta_2:=\mean{|\xi|^2\onesub{|\xi|\leq \tau_z}}$ with $\xi$ being a standard normal distribution. It follows that 
\begin{align*}
\bw^\top\lb\frac{1}{m}\sum_{k=1}^m|\ba_k^\top\bz|^2\ba_k\ba_k^\top\onesub{|\ba_k^\top\bz|\leq\tau_z\|\bz\|}\rb\bw\leq (\vartheta_2+\rho_1)\|\bz\|^2\|\bw\|^2+
\vartheta_1|\bz^\top\bw|^2.
\end{align*}
Therefore, for all $\bz\in\R^n$ and $\BW\in T_{\bz}$, one has 
\begin{align*}
\frac{1}{m}\|\A_{\bz}(\BW)\|^2&\leq 4(\vartheta_2+\rho_1)\|\bz\|^2\|\bw\|^2+4
\vartheta_1|\bz^\top\bw|^2\\
&=2(\vartheta_2+\rho_1)\lb 2\|\bz\|^2\|\bw\|^2+2|\bz^\top\bw|^2\rb
+4(\vartheta_1-\vartheta_2-\rho_1)|\bw^\top\bx|^2\\
&=2(\vartheta_2+\rho_1)\|\BW\|^2_F+4(\vartheta_1-\vartheta_2-\rho_1)|\bw^\top\bx|^2.
\end{align*}
The upper bound is then obtained by further noting that 
if $ \vartheta_1- \vartheta_2-\rho_1\geq 0$, then 
$$4( \vartheta_1- \vartheta_2-\rho_1)|\bz^\top\bw|^2\leq ( \vartheta_1- \vartheta_2-\rho_1)\lb 2\|\bz\|^2\|\bw\|^2+2|\bz^\top\bw|^2\rb=( \vartheta_1- \vartheta_2-\rho_1)\|\BW\|_F^2.$$
%%%
\paragraph{Lower bound} To establish the lower bound, first observe that
\begin{align*}
\dsone_{\ME_1^k(\bx)\cap\ME_1^k(\bz)\cap\ME_2^k(\bz)} &= \dsone_{\ME_1^k(\bz)}-\dsone_{\ME_1^k(\bz)}\dsone_{(\ME_2^k(\bz))^c\cup(\ME_1^k(\bx))^c}\geq \dsone_{\ME_1^k(\bz)}-\dsone_{\ME_1^k(\bz)}\dsone_{(\ME_2^k(\bz))^c}-\dsone_{\ME_1^k(\bz)}\dsone_{(\ME_1^k(\bx))^c}\\
&\geq \dsone_{\ME_1^k(\bz)}-\dsone_{\ME_1^k(\bz)}\dsone_{\lcb|\ba_k^\top\bh|>1.15\tau_h\|\bh\|\rcb}-\dsone_{\ME_1^k(\bz)}\dsone_{|\lab\ba_k^\top\bx|>0.9\tau_x\|\bx\|\rcb}
\end{align*}
where $\bh=\bz-\bx$, and in the last inequality we have used Lemmas~\ref{prop:events_prop_x} and \ref{prop:events_prop_z}.
{Note that by Lemma~\ref{lem:vector_bd_matrix}, the assumption $\|\bh\|\leq \frac{1}{11}\|\bz\|$ for Lemma~\ref{prop:events_prop_z} holds provided $\|\bz\bz^\top-\bx\bx^\top\|_F\leq \frac{1}{13}\|\bx\|^2$.} It follows that
\begin{align*}
\frac{1}{m}\ln\A_{\bz}(\BW)\rn^2 & =\frac{4}{m}\sum_{k=1}^m|\ba_k^\top\bz|^2|\ba_k^\top\bw|^2\dsone_{\ME_1^k(\bx)\cap\ME_1^k(\bz)\cap\ME_2^k(\bz)}\\
&\geq \frac{4}{m}\sum_{k=1}^m|\ba_k^\top\bz|^2|\ba_k^\top\bw|^2\onesub{|\ba_k^\top\bz|\leq\tau_z\|\bz\|}\\
&-\frac{4}{m}\sum_{k=1}^m|\ba_k^\top\bz|^2|\ba_k^\top\bw|^2\onesub{|\ba_k^\top\bz|\leq\tau_z\|\bz\|}\onesub{|\ba_k^\top\bh|\geq 1.15\tau_h\|\bh\|}\\
&-\frac{4}{m}\sum_{k=1}^m|\ba_k^\top\bz|^2|\ba_k^\top\bw|^2\onesub{|\ba_k^\top\bz|\leq\tau_z\|\bz\|}\onesub{|\ba_k^\top\bx|\geq 0.9\tau_x\|\bx\|}\\
&:=\TM_1-\TM_2-\TM_3.
\end{align*}
Next we will provide a lower bound for $\TM_1$ and upper bounds for $\TM_2$
and $\TM_3$.
By \eqref{eq:truncated_iso}, $\TM_1$ can be bounded from below as 
\begin{align*}
\TM_1\geq 4\vartheta_1|\bz^\top\bw|^2+4(\vartheta_1-\rho_1)\|\bz\|^2\|\bw\|^2.
\end{align*}
An upper bound for $\TM_2$ can be established as follows:
\begin{align*}
\TM_2&\leq 4\tau_z^2\|\bz\|^2\lb\frac{1}{m}\sum_{k=1}^m|\ba_k^\top\bw|^2\onesub{|\ba_k^\top\bh|\geq 1.15\tau_h\|\bh\|}\rb\\
&\leq 4\tau_z^2\lb 5.75\tau_he^{-0.64\tau_h^2}+\epsilon\rb\|\bw\|^2\|\bz\|^2,
\end{align*}
where the second inequality holds with probability exceeding $1-e^{-\Omega(m\epsilon^2)}$ provided $m\gtrsim \epsilon^{-2}\log\epsilon^{-1}\cdot n$; see
Lemma~\ref{lem:key1}. Similarly, $\TM_3$ can be bounded from above as 
\begin{align*}
\TM_3\leq 4\tau_z^2\lb 3.6\tau_xe^{-0.39\tau_x^2}+\epsilon\rb\|\bw\|^2\|\bz\|^2.
\end{align*}

Combining the bounds for $\TM_1$, $\TM_2$ and $\TM_3$ together yields that
\begin{align*}
\frac{1}{m}\ln\A_{\bz}(\BW)\rn^2&\geq 4 \vartheta_1|\bz^\top\bw|^2+4\lb \vartheta_2-\rho_1-\rho_2\rb\|\bz\|^2\|\bw\|^2\\
&\geq 2\min\lcb \vartheta_1, \vartheta_2-\rho_1-\rho_2\rcb\|\BW\|^2_F,
\end{align*}
where $\rho_2=6\tau_z^2\tau_he^{-0.64\tau_h^2}+4\tau_z^2\tau_xe^{-0.39\tau_x^2}+o(1)$ by taking $\epsilon$ to be sufficiently small.
%%%%%
\end{proof}

%===========
\subsection{Proof of Theorem~\ref{thm:weak_correlation}}
By Lemma~\ref{lem:vector_bd_matrix}, it suffices to establish the bound for 
\begin{align*}
\frac{\ln\bh\rn}{\ln\bz\rn}\leq \min\lcb\sqrt{\frac{\rho_3}{3\lb \tau_z^4+5\tau_z^3+8\tau_z^2+2\tau_h^2\rb}},\frac{\rho_3}{15\tau_z\tau_{h,z}},\frac{1}{11}\rcb.\numberthis\label{eq:basin_size1}
\end{align*}
By Theorem~\ref{thm:well_conditioned}, we have 
\begin{align*}
\ln\frac{1}{\sqrt{m}}\P_{T_{\bz}}\A_{\bz}^\top\rn
&=\sup_{\|\bb\|=1,\|\BW\|_F=1}\frac{1}{\sqrt{m}}\lab\la\BW,\P_{T_{\bz}}\A_{\bz}^\top(\bb)\ra\rab\\
&=\sup_{\|\bb\|=1,\|\BW\|_F=1}\frac{1}{\sqrt{m}}\lab\la\A_{\bz}\P_{T_{\bz}}(\BW),\bb\ra\rab\\
&\leq \sqrt{\btup}.\numberthis\label{eq:upper_lower_1}
\end{align*}
Thus, it remains to establish an upper bound for $\frac{1}{\sqrt{m}}\ln\A_{\bz}(\I-P_{T_{\bz}})(\bz\bz^\top-\bx\bx^\top)\rn$. Noting that 
\begin{align*}
(\I-P_{T_{\bz}})(\bz\bz^\top-\bx\bx^\top)& =-\lb \I-\frac{\bz\bz^\top}{\|\bz\|^2}\rb\bx\bx^\top\lb \I-\frac{\bz\bz^\top}{\|\bz\|^2}\rb=-\lb\bx-\frac{\bz^\top\bx}{\|\bz\|^2}\bz\rb\lb\bx-\frac{\bz^\top\bx}{\|\bz\|^2}\bz\rb^\top
\end{align*}
and letting $\bh=\bz-\bx$, we have 
\begin{align*}
(\I-P_{T_{\bz}})(\bz\bz^\top-\bx\bx^\top) = -\lb \frac{\bz^\top\bh}{\|\bz\|^2}\bz-\bh\rb\lb \frac{\bz^\top\bh}{\|\bz\|^2}\bz-\bh\rb^\top.
\end{align*}
It follows that (cf. Definition~\eqref{eq:Az})
\begin{align*}
&\frac{1}{m}\ln\A_{\bz}(\I-P_{T_{\bz}})(\bz\bz^\top-\bx\bx^\top)\rn^2\\&=\frac{1}{m}\sum_{k=1}^m\lb\ba_k^\top\lb \frac{\bz^\top\bh}{\|\bz\|^2}\bz-\bh\rb\rb^4\dsone_{\ME_1^k(\bx)\cap \ME_1^k(\bz)\cap\ME_2^k(\bz)}\\
&\leq\lab \frac{1}{m}\sum_{k=1}^m\frac{(\ba_k^\top\bz)^4(\bz^\top\bh)^4}{\|\bz\|^8}\dsone_{\ME_1^k(\bx)\cap \ME_1^k(\bz)\cap\ME_2^k(\bz)}\rab\\
&+\lab \frac{4}{m}\sum_{k=1}^m\frac{(\bz^\top\bh)^3}{\|\bz\|^6}(\ba_k^\top\bz)^3(\ba_k^\top\bh)\dsone_{\ME_1^k(\bx)\cap \ME_1^k(\bz)\cap\ME_2^k(\bz)}\rab\\
&+\lab \frac{6}{m}\sum_{k=1}^m\frac{(\bz^\top\bh)^2}{\|\bz\|^4}(\ba_k^\top\bz)^2(\ba_k^\top\bh)^2\dsone_{\ME_1^k(\bx)\cap \ME_1^k(\bz)\cap\ME_2^k(\bz)}\rab\\
&+\lab \frac{4}{m}\sum_{k=1}^m\frac{(\bz^\top\bh)}{\|\bz\|^2}(\ba_k^\top\bz)(\ba_k^\top\bh)^3\dsone_{\ME_1^k(\bx)\cap \ME_1^k(\bz)\cap\ME_2^k(\bz)}\rab\\
&+\lab \frac{1}{m}\sum_{k=1}^m(\ba_k^\top\bh)^4\dsone_{\ME_1^k(\bx)\cap \ME_1^k(\bz)\cap\ME_2^k(\bz)}\rab\\
&:=\TM_1+\TM_2+\TM_3+\TM_4+\TM_5.
\end{align*}
Next we will bound $\TM_i,~i=1,\cdots,5$ one after another.
%%%%
\paragraph{Upper bound of $\TM_1$} Direct calculation yields that 
\begin{align*}
\TM_1\leq  \frac{1}{m}\sum_{k=1}^m\frac{|\bz^\top\bh|^4}{\|\bz\|^8}|\ba_k^\top\bz|^4\dsone_{\ME_1^k(\bz)}\leq\tau_{z}^4\|\bh\|^4.
\end{align*}
%%%%
\paragraph{Upper bound of $\TM_2$} It is evident that 
\begin{align*}
\TM_2&\leq  \frac{4}{m}\sum_{k=1}^m\frac{|\bz^\top\bh|^3}{\|\bz\|^6}|\ba_k^\top\bz|^3|\ba_k^\top\bh|\dsone_{\ME_1^k(\bx)\cap \ME_1^k(\bz)\cap\ME_2^k(\bz)}\\
&\leq \frac{4}{m}\sum_{k=1}^m\frac{|\bz^\top\bh|^3}{\|\bz\|^6}|\ba_k^\top\bz|^3|\ba_k^\top\bh|\dsone_{\ME_1^k(\bz)}\\
&\leq4\tau_{z}^3\|\bh\|^3\lb\frac{1}{m}\sum_{k=1}^m|\ba_k^\top\bh|\rb\\
&\leq 4\tau_{z}^3\|\bh\|^3\sqrt{\frac{1}{m}\sum_{k=1}^m|\ba_k^\top\bh|^2}\\
&\leq 5\tau_{z}^3\|\bh\|^4,
\end{align*}
where the last inequality holds with probability at least $1-e^{-\Omega(m)}$ provided $m\gtrsim n$; see \eqref{eq:event_04}.
%%%%
\paragraph{Upper bound of $\TM_3$} Similarly, we have 
\begin{align*}
\TM_3&\leq \lab \frac{6}{m}\sum_{k=1}^m\frac{(\bz^\top\bh)^2}{\|\bz\|^4}(\ba_k^\top\bz)^2(\ba_k^\top\bh)^2\dsone_{\ME_1^k(\bx)\cap \ME_1^k(\bz)\cap\ME_2^k(\bz)}\rab\\
&\leq6\tau_z^2\|\bh\|^2\frac{1}{m}\sum_{k=1}^m|\ba_k^\top\bh|^2\\
&\leq 8\tau_z^2\|\bh\|^4.
\end{align*}
%%%%
\paragraph{Upper bound of $\TM_4$} By Lemma~\ref{prop:events_prop_z}, when $\|\bh\|\leq\frac{1}{11}$, we have 
\begin{align*}
\ME_1^k(\bx)\cap \ME_1^k(\bz)\cap\ME_2^k(\bz)\subset\ME_1^k(\bx)\cap \ME_1^k(\bz)\cap\ME_4^k(\bz)
\subset\lcb|\ba_k^\top\bh|\leq\tau_{h,z}\|\bz\|\rcb.\numberthis\label{eq:tmptmp}
\end{align*}
It follows that $\dsone_{\ME_1^k(\bx)\cap \ME_1^k(\bz)\cap\ME_2^k(\bz)}\leq \onesub{|\ba_k^\top\bh|\leq\tau_{h,z}\|\bz\|}$ and
\begin{align*}
\TM_4&\leq  \frac{4}{m}\sum_{k=1}^m\frac{|\bz^\top\bh|}{\|\bz\|^2}|\ba_k^\top\bz||\ba_k^\top\bh|^3\dsone_{\ME_1^k(\bx)\cap \ME_1^k(\bz)\cap\ME_2^k(\bz)}\\
&\leq 4\tau_z\|\bh\|\lb\frac{1}{m}\sum_{k=1}^m|\ba_k^\top\bh|^3\dsone_{\ME_1^k(\bx)\cap \ME_1^k(\bz)\cap\ME_2^k(\bz)}\rb\\
&\leq4\tau_z\|\bh\|\lb\frac{1}{m}\sum_{k=1}^m|\ba_k^\top\bh|^3 \onesub{|\ba_k^\top\bh|\leq\tau_{h,z}\|\bz\|}\rb\\
&\leq 4\tau_z\tau_{h,z}\|\bh\|\|\bz\|\lb\frac{1}{m}\sum_{k=1}^m|\ba_k^\top\bh|^2 \rb\\
&\leq5\tau_z\tau_{h,z}\|\bh\|^3\|\bz\|.
\end{align*}
%%%%
\paragraph{Upper bound of $\TM_5$} Using \eqref{eq:tmptmp} again gives 
\begin{align*}
\TM_5&\leq \frac{1}{m}\sum_{k=1}^m|\ba_k^\top\bh|^4\onesub{|\ba_k^\top\bh|\leq\tau_{h,z}\|\bz\|}\\
&=\frac{1}{m}\sum_{k=1}^m|\ba_k^\top\bh|^4\onesub{\{|\ba_k^\top\bh|\leq\tau_{h,z}\|\bz\|\}\cap\ME_5^k(\bz)}+
\frac{1}{m}\sum_{k=1}^m|\ba_k^\top\bh|^4\onesub{\{|\ba_k^\top\bh|\leq\tau_{h,z}\|\bz\|\}\cap\lb\ME_5^k(\bz)\rb^c}\\
&\leq \frac{1}{m}\sum_{k=1}^m|\ba_k^\top\bh|^4\onesub{|\ba_k^\top\bh|\leq 1.15\tau_h\|\bh\|}+\frac{1}{m}\sum_{k=1}^m|\ba_k^\top\bh|^4\onesub{\{|\ba_k^\top\bh|\leq\tau_{h,z}\|\bz\|\}\cap\{|\ba_k^\top\bh|>1.15\tau_h\|\bh\|\}}\\
&\leq 1.15\|\bh\|^2\tau_h^2\lb\frac{1}{m}\sum_{k=1}^m|\ba_k^\top\bh|^2\rb+
\tau_{h,z}^2\|\bz\|^2\bh^\top\lb\frac{1}{m}\sum_{k=1}^m\ba_k\ba_k^\top\onesub{|\ba_k^\top\bh|>1.15\tau_h\|\bh\|}\rb\bh\\
&\leq 2\tau_h^2\|\bh\|^4+\lb6\tau_h\tau_{h,z}^2e^{-0.64\tau_h^2}+o(1)\rb\|\bz\|^2\|\bh\|^2,
\end{align*}
where {we have used  Lemma~\ref{lem:key2} in the last line by choosing sufficiently small $\epsilon$ for fixed $\tau_{h,z}^2$.}\\

Putting it all together, we have 
\begin{align*}
&\frac{1}{m}\ln\A_{\bz}(\I-\P_{T_{\bz}})(\bz\bz^\top-\bx\bx^\top)\rn^2\\
&\leq\lb \lb \tau_z^4+5\tau_z^3+8\tau_z^2+2\tau_h^2\rb\frac{\ln\bh\rn^2}{\ln\bz\rn^2}+5\tau_z\tau_{h,z}\frac{\ln\bh\rn}{\ln\bz\rn}+\lb 6\tau_h\tau_{h,z}^2e^{-0.64\tau_h^2}+o(1)\rb\rb\ln\bz\rn^2\ln\bh\rn^2\\
&\leq 1.25\lb \lb \tau_z^4+5\tau_z^3+8\tau_z^2+2\tau_h^2\rb\frac{\ln\bh\rn^2}{\ln\bz\rn^2}+5\tau_z\tau_{h,z}\frac{\ln\bh\rn}{\ln\bz\rn}+\lb 6\tau_h\tau_{h,z}^2e^{-0.64\tau_h^2}+o(1)\rb\rb\ln\bz\bz^\top-\bx\bx^\top\rn_F^2,
\end{align*}
where the last line uses Lemma~\ref{lem:vector_bd_matrix}. 
The proof is complete after combing this bound with \eqref{eq:upper_lower_1} and making proper substitutions.
\section{Conclusion and future directions}\label{sec:conclusion}
We have presented a Riemannian gradient descent algorithm and its truncated variant for solving systems of phaseless equations. Exact recovery guarantee has been established for the truncated variant, showing that the algorithm is able to achieve successful recovery with the optimal sampling complexity. In addition, empirical evaluations show that our algorithm are competitive with
other state-of-the-art first order methods. We conclude this paper by pointing out a few problems for future directions:
\begin{itemize}
\item The Riemannian gradient descent algorithms studied in this paper can be easily extended to Riemannian conjugate gradient descent algorithms which should be substantially faster. Theoretical analysis of the conjugate gradient descent type algorithms is an interesting direction for future research. 
\item Numerical simulations show that RGrad can be similarly effective provided an appropriate stepsize is used, which suggests the possibility of analyzing this vanilla Riemannian gradient descent algorithm. The leave-one-out technique that has been employed in \cite{Chen_implicit} provides a potential tool for the analysis. It is also worth investigating the convergence of the algorithms under other measurement models.
\item The algorithms presented in this paper applies equally to the problem of reconstructing a  rank-$r$ positive semidefinite matrices from rank-$1$ projection measurements, namely, solving the following systems of equations:
\begin{align*}
y_k = \langle\ba_k\ba_k^\top,\BX\rangle,\quad k=1,\cdots, m,
\end{align*}
where the unknown matrix $\BX$ is positive semidefinite and is of rank $r$. It may also be possible to establish the convergence of the algorithms for this setting. 
\item As mentioned in the introduction, geometric landscape for the problem of  solving systems of phaseless equations has been studied in \cite{SQW:FCM:18}. Notice that the analysis there is carried out in Euclidean space after parameterization. More precisely, if we submit $\BZ=\bz\bz^\top$ into \eqref{eq:low_rank}, the constraints can be removed and the reconstruction problem can be achieved by minimizing the following function:
\begin{align*}
f(\bz) = \sum_{k=1}^m\lb(\ba_k^\top\bz)^2-\by\rb^2.
\end{align*}
It was shown in \cite{SQW:FCM:18} that under the Gaussian measurement model $f$ does not have a spurious local minima provided $m\gtrsim n\log^3 n$. In a different direction, one can investigate the geometric landscape of the problem directly on manifold. In particular, it is of great interest to study the geometric property of 
\begin{align*}
f(\BZ) = \|\A(\BZ)-\by\|^2
\end{align*}
over the embedded manifold of rank-$1$ positive semidefinite matrices.  Progress torwards this direction will be reported separately.\end{itemize}
%-------------------------------------
\section*{Acknowledgments}
KW would like to thank Wen Huang for a fruitful discussion about Riemannian geometry. 
%-------------------------------------
%\bibliographystyle{unsrt}
\bibliographystyle{abbrv}
\bibliography{ref}
\end{document}